\documentclass[12pt,reqno]{amsart}
\pdfoutput=1
\usepackage{latexsym, amsmath, amssymb, amsthm, bm}
\usepackage{rsfso}
\usepackage[mathscr]{eucal}
\usepackage{mathtools}
\usepackage{paralist}
\usepackage[alphabetic]{amsrefs}
\usepackage[T1]{fontenc}
\usepackage{mathptmx}
\usepackage{microtype}
\usepackage[colorlinks=true, linkcolor=blue, urlcolor=blue, citecolor=blue]%
    {hyperref}
\usepackage[centering, includeheadfoot, hmargin=1.0in, tmargin=1.0in, 
  bmargin=1in, headheight=6pt]{geometry}
\newtheorem{lemma}{Lemma}[section]
\newtheorem{theorem}[lemma]{Theorem}
\newtheorem{corollary}[lemma]{Corollary}
\newtheorem{proposition}[lemma]{Proposition}
\theoremstyle{definition}
\newtheorem{definition}[lemma]{Definition}
\newtheorem{remark}[lemma]{Remark}
\newtheorem{examplex}[lemma]{Example}
\newenvironment{example}
  {\pushQED{\qed}\examplex}
  {\popQED\endexamplex}
\newcommand{\define}[1]{{\bfseries\itshape #1}}
  
\newcommand{\abs}[1]{\ensuremath{\left| #1 \right|}}
\newcommand{\ideal}[1]{\ensuremath{\left\langle #1 \right\rangle}}
\newcommand{\transpose}{\ensuremath{\textsf{\upshape T}}} 
\DeclareMathOperator{\codim}{codim}
\DeclareMathOperator{\conv}{conv}
\DeclareMathOperator{\green}{\textit{a}}
\DeclareMathOperator{\Hom}{Hom}
\DeclareMathOperator{\kept}{\lambda}
\DeclareMathOperator{\len}{\ell}
\DeclareMathOperator{\py}{py}
\DeclareMathOperator{\qp}{qp}
\DeclareMathOperator{\qr}{qr}
\DeclareMathOperator{\Sos}{\Sigma}
\DeclareMathOperator{\Span}{Span}
\DeclareMathOperator{\Tor}{Tor}
\begin{document}

\vspace*{-0.75em}

\title[Sums of Squares and Quadratic Persistence]%
  {Sums of Squares and Quadratic Persistence\\ on Real Projective Varieties}

\author[G.~Blekherman]{Grigoriy Blekherman}
\address{Grigoriy Blekherman: School of Mathematics, Georgia Tech, 686 Cherry
  Street, Atlanta, Georgia, 30332, United States of America;
  {\normalfont \texttt{greg@math.gatech.edu}}}

\author[R.~Sinn]{Rainer Sinn}
\address{Rainer Sinn: Institut f\"ur Mathematik, Arnimallee 2, Freie
  Universit\"at Berlin, 14195 Berlin, Germany; {\normalfont
    \texttt{rsinn@zedat.fu-berlin.de}}}

\author[G.G.~Smith]{Gregory G.{} Smith}
\address{Gregory G.{} Smith: Department of Mathematics and Statistics, Queen's
  University, Kingston, Ontario, K7L 3N6, Canada; {\normalfont
    \texttt{ggsmith@mast.queensu.ca}}}

\author[M.~Velasco]{Mauricio Velasco} 
\address{Mauricio Velasco: Departamento de Matem\'aticas, Universidad de los
  Andes, Carrera 1 No. 18a 10, Edificio~H, Primer Piso, 111711 Bogot\'a,
  Colombia; {\normalfont \texttt{mvelasco@uniandes.edu.co}}}


% \date{2020-04-13}

\begin{abstract}
  We bound the Pythagoras number of a real projective subvariety: the smallest
  positive integer $r$ such that every sum of squares of linear forms in its
  homogeneous coordinate ring is a sum of at most $r$ squares.  Enhancing
  existing methods, we exhibit three distinct upper bounds involving known
  invariants.  In contrast, our lower bound depends on a new invariant of a
  projective subvariety called the quadratic persistence.  Defined by
  projecting away from points, this numerical invariant is closely related to
  the linear syzygies of the variety.  In addition, we classify the projective
  subvarieties of maximal and almost-maximal quadratic persistence, and
  determine their Pythagoras numbers.
\end{abstract}

\maketitle

%%%%%%%%%%%%%%%%%%%%%%%%%%%%%%%%%%%%%%%%%%%%%%%%%%%%%%%%%%%%%%%%%%%%%%%%%%%%%%
\addtocounter{section}{1}
% \section{Overview}

\vspace*{-1.0em}

\noindent
Sums of squares occupy a central place in real algebraic geometry,
optimization, and number theory.  Being able to represent an element in a
commutative ring as a sum of squares has substantial ramifications in the
study of non-negativity and quadratic forms, whereas the constructive aspects
of these representations are indispensable in developing efficient
computational tools.  The Pfister Theorem \cite{Lam}*{Corollary~XI.4.11},
proving that any nonnegative rational function in the field
$\mathbb{R}(x_0, x_1, \dotsc, x_n)$ is a sum of at most $2^{n+1}$ squares,
serves as a motivational example.  Our primary objective is to find similar
effective bounds for homogeneous elements in a real affine algebra and our
approach exposes some unexpected connections between real and complex
geometry.

Given a real subvariety $X \subseteq \mathbb{P}^n$, let $I_X$ be its saturated
homogeneous ideal in the polynomial ring
$S \coloneqq \mathbb{R}[x_0, x_1, \dotsc, x_n]$ and let $R \coloneqq S/I_X$
denote its homogeneous coordinate ring.  Inspired by
\cite{Lam}*{Section~XIII.5}, the \define{Pythagoras number} $\py(X)$ is the
smallest positive integer $r$ such that any sum of squares of linear forms in
$R$ can be expressed as the sum of at most $r$ squares.  We focus on
homogeneous polynomials of degree $2$ because the appropriate Veronese
re-embedding of $X$ reduces the analysis to this case.  Refining and
consolidating existing methods, our first theorem provides three different
upper bounds on the Pythagoras number for real projective subvarieties.  To
articulate this result, we set $\green(X)$ to be the largest number $k$ such
that the homogeneous ideal $I_X$ is generated by quadratic polynomials and the
first $k-1$ maps in its minimal free resolution are represented by matrices of
linear forms, and we refer to a projective subvariety
$X' \subseteq \mathbb{P}^n$ as $2$\nobreakdash-regular if its homogeneous
ideal $I_{X'}$ is generated by quadratic polynomials and all the maps in its
minimal free resolution are represented by matrices of linear forms; see
\cite{Eis}*{Sections 4A and 8D}.

\begin{theorem}
  \label{t:1}
  For any real subvariety $X \subseteq \mathbb{P}^n$ such that the set
  $X(\mathbb{R})$ of real points is not contained in a hyperplane, we have the
  following upper bounds:
  \begin{compactenum}[\upshape i.]
  \item $\binom{\py(X) +1}{2} < \dim_{\mathbb{R}} R_2$;
  \item $\py(X) \leqslant n + 1 - \min \{ \green(X), \codim(X) \}$;
  \item $\py(X)$ is at most one more than the dimension of any real
    $2$\nobreakdash-regular variety containing $X$.
  \end{compactenum}
\end{theorem}

\noindent
Together Example~\ref{e:genCan}, Example~\ref{e:PP2Up}, and
Example~\ref{e:Peter} illustrate that any one of these upper bounds can be
stronger than the other two.

More significantly, we devise a lower bound for the Pythagoras number of a
real subvariety.  The key is to introduce a new numerical invariant for a
complex projective subvariety.  For any nonnegative integer $k$ and any subset
$\Gamma$ of $k$ closed points in $X$, let
$\pi_{\Gamma} \colon \mathbb{P}^n \dashrightarrow \mathbb{P}^{n-k}$ be the
rational map given by the linear projection away from $\Gamma$.  The
\define{quadratic persistence} $\qp(X)$ of the subvariety
$X \subseteq \mathbb{P}^n$ is the smallest nonnegative integer $k$ for which
there exists a subset $\Gamma$ of $k$ closed points in $X$ such that the
homogeneous ideal $I_{\pi_{\Gamma}(X)}$ contains no quadratic polynomials.
Lemma~\ref{l:qpGen} shows that the quadratic persistence of an irreducible
variety may be calculated by projecting away from a general set of closed
points and Lemma~\ref{l:qpBasic} establishes the basic inequality
$\qp(X) \leqslant \codim(X)$.  On the other hand, quadratic persistence is
also intimately related to linear syzygies.  To state our second major result,
let $\len(X)$ be the number of nonzero entries in the first row of the Betti
table for the homogeneous coordinate ring $R$ regarded as an
$S$\nobreakdash-module; see \eqref{d:len} or \cite{Eis}*{Section~8D}.

\begin{theorem}
  \label{t:2}
  For a non-degenerate irreducible complex subvariety
  $X \subseteq \mathbb{P}^n$, we have $\qp(X) \geqslant \len(X)$.
\end{theorem}

\noindent
By replacing codimension with quadratic persistence, this theorem sharpens the
first part of Green's
$K_{p,1}$\nobreakdash-Theorem~\cite{Gre}*{Theorem~3.c.1}.  Even better, we use
quadratic persistence to calculate $\ell(X)$ in some situations; see
Proposition~\ref{p:qpEq} and Proposition~\ref{p:prism}.  Fulfilling our
original motivation for introducing quadratic persistence, our third theorem
gives a lower bound on the Pythagoras number of a real projective variety that
does not lie in a hyperplane and contains a nonsingular real point.

\begin{theorem}
  \label{t:3}
  For any non-degenerate irreducible totally-real subvariety $X \subseteq
  \mathbb{P}^n$, we have
  \[
    \py(X) \geqslant n + 1 - \qp(X) \geqslant 1 + \dim(X) \, .
  \]
\end{theorem}

\noindent
Although the Pythagoras number is a semi-algebraic invariant relying on the
real structure, the lower bounds are algebraic invariants depending only on
the complex geometry of the subvariety.

Counterintuitively, our upper and lower bounds on the Pythagoras number agree
when the quadratic persistence is relatively large.  In the maximal case, our
fourth theorem strengthens Theorem~1.1 in \cite{BPSV} and yields yet another
characterization for varieties of minimal degree.

\begin{theorem}
  \label{t:4}
  For any non-degenerate irreducible totally-real subvariety
  $X \subseteq \mathbb{P}^n$, the following conditions are equivalent:
  \begin{compactenum}[\upshape a.]
  \item $\qp(X) = \codim(X)$;
  \item $\py(X) = 1 + \dim(X)$;
  \item $\deg(X) = 1 + \codim(X)$.
  \end{compactenum}
\end{theorem}

\noindent
In the nearly-maximal case, we can also compute the Pythagoras number and
classify the varieties under the additional hypothesis that the homogeneous
coordinate ring is Cohen--Macaulay.

\begin{theorem}
  \label{t:5}
  Let $X \subseteq \mathbb{P}^n$ denote a non-degenerate irreducible
  totally-real subvariety.  If $X$ is arithmetically Cohen--Macaulay, then the
  following conditions are equivalent:
  \begin{compactenum}[\upshape a.]
  \item $\qp(X) = \codim(X) - 1$;
  \item $\py(X) = 2 + \dim(X)$;
  \item $\deg(X) = 2 + \codim(X)$ or $X$ is a codimension-one subvariety of a
    variety of minimal degree.
  \end{compactenum}
\end{theorem}

\noindent
This fifth result is a counterpart to the third part of Green's
$K_{p,1}$\nobreakdash-Theorem~\cite{Gre}*{Theorem~3.c.1} where quadratic
persistence supplants the degree of a morphism. More directly,
Theorem~\ref{t:2} and Theorem~\ref{t:3} establish the second part of Green's
$K_{p,1}$\nobreakdash-Theorem for totally-real projective varieties.  Many of
the implications between the three conditions in both Theorem~\ref{t:4} and
Theorem~\ref{t:5} continue to hold under weaker assumptions on the subvariety
$X$; see Section~\ref{s:qp}.

%% ---------------------------------------------------------------------------
\subsection*{Explicit Bounds in Special Cases}

To better assess the power of our geometric approach, we produce concrete
bounds on the Pythagoras numbers and the quadratic persistence for projective
curves and toric subvarieties.  Corollary~\ref{c:canon} proves that the
Pythagoras number for a canonical real curve is bounded above by its real
gonality: the lowest degree of a real non-constant morphism from the curve to
the real projective line.  Using Green's Conjecture~\cite{Voi} for a general
canonical curve $X \subset \mathbb{P}^{g-1}$, Example~\ref{e:genCan}
specializes the bounds from Theorem~\ref{t:1} and Example~\ref{e:strict} shows
that the quadratic persistence is strictly larger than the number $\len(X)$ of
nonzero entries in the first row of the Betti table for its homogeneous
coordinate ring.  Similarly, for a high-degree curve $X \subset \mathbb{P}^n$,
Corollary~\ref{c:gon} bounds the Pythagoras number via its gonality and, using
the Gonality Conjecture~\cite{EL}, Example~\ref{e:high} establishes that
$\qp(X) > \ell(X)$.  The close relationship between the Pythagoras number and
these other sophisticated numerical invariants is remarkable, and the
observations about quadratic persistence answer Question~5.8 in \cite{HK15}
negatively.

When compared to curves, the proofs of the analogous bounds for projective
toric subvarieties reverse the flow of information.  Instead of the well-known
invariant for curves, Corollary~\ref{c:lines} proves that the Pythagoras
number of a projective toric subvariety
$X_{P \cap \mathbb{Z}^d} \subseteq \mathbb{P}^n$ is bounded above by a simple
new invariant: the minimal number of parallel lines needed to cover all of the
lattice points in the polytope $P \subset \mathbb{R}^d$.  For toric surfaces,
this invariant---disguised as the lattice width of a polygon---is already
related to linear syzygies; see Definition~1.5 and Conjecture~1.6 in
\cite{CCDL}.  Proposition~\ref{p:prism} computes the quadratic persistence for
any toric subvariety associated to a tall prism (the product a lattice
polytope and a sufficiently long interval) and, thereby, deduces both the
Pythagoras number and the number of nonzero entries in the first row of the
Betti table.  In contrast with our examples for curves, we have
$\qp(X) = \len(X)$ in this situation.  Example~\ref{e:segVer} showcases a
family for which the hypothesis on the height of the prism is vacuous.  As
Section~6 in \cite{Rai} underscores the difficulty in describing the linear
syzygy modules for the Segre–Veronese embeddings of a product of projective
spaces, our numerical success with the toric subvarieties associated to a tall
prisms is all the more surprising.

%% ---------------------------------------------------------------------------
\subsection*{Pythagoras Numbers in Applications}

By emphasizing projective subvarieties, our approach unifies various
viewpoints.  Hilbert's Theorem~\cite{Hil}, demonstrating that every
nonnegative ternary quartic is the sum of $3$ squares, is the primal source
for the Pythagoras numbers of homogeneous elements in a real affine algebra.
From our perspective, this is the same as showing that the Pythagoras number
of the Veronese surface $\mathbb{P}^2 \subset \mathbb{P}^5$ equals $3$; see
Example~\ref{e:PP2Up}.  Unlike the intervening work on rational functions,
\cite{CLR} again concentrates on homogeneous polynomials, providing both lower
and upper bounds on their Pythagoras numbers.  Advancing these ideas,
\cite{Sch17} establishes new lower bounds that are much closer to the existing
upper bounds.  By re-proving Theorem~3.6 in \cite{Sch17},
Example~\ref{e:PP2Low} hints at the universality of our geometric paradigm.

In optimization, the Pythagoras number is typically recast in terms of the
rank of Gram matrices.  Each quadratic form
$f \in S \coloneqq \mathbb{R}[x_0, x_1, \dotsc, x_n]$ corresponds to a real
symmetric matrix $\mathbf{A}$ such that
$f = \bm{x}^{\transpose} \mathbf{A} \, \bm{x}$, where $\bm{x}$ is the column
vector whose entries are the variables $x_0, x_1, \dotsc, x_n$.  By the
spectral theorem, the quadratic form $f$ is a sum of squares if and only if
the matrix $\mathbf{A}$ is positive semidefinite. As a consequence, the set of
sum-of-squares representations for a quadratic form in any real affine algebra
is the intersection of the cone of positive-semidefinite matrices with an
affine linear space, called the Gram spectrahedron.  Hence, deciding whether a
quadratic form is a sum of squares is equivalent to the feasibility of a
semidefinite programming problem.  Better yet, the polynomial $f \in S_2$ is a
sum of $r$ squares if and only if the positive-semidefinite matrix
$\mathbf{A}$ has rank $r$.  Thus, the Pythagoras number is the maximum rank
among matrices of minimal rank in the Gram spectrahedra.  Computationally,
upper bounds on this Pythagoras number allow \cite{BM} to improve the
scalability of such optimization problems by factoring the matrix $\mathbf{A}$
as $\mathbf{B} \, \mathbf{B}^{\transpose}$, where $\mathbf{B}$ is a real
$(n + 1) \times r$\nobreakdash-matrix.  Although this surrogate destroys
convexity, \cites{BVB0, BVB} show that local methods do reliably converge to
global optima.  Beyond the upper bounds in Theorem~\ref{t:1},
Theorems~\ref{t:4}--\ref{t:5} can both be reinterpreted as structural results
about Gram spectrahedra.  For instance, Corollary~\ref{c:acmUp} recovers
Theorem~3.5 in \cite{CPSV}.

Pythagoras numbers also appear in the study of metric embeddings of graphs.
As explained in Section~1 of \cite{LV}, deciding whether a simple graph has an
isometric embedding into the $(r - 1)$\nobreakdash-dimensional spherical
metric space is tantamount to solving a matrix completion problem.
Specifically, given a graph $G$ with $n + 1$ vertices, one seeks the smallest
number $r \in \mathbb{N}$ such that, for any positive-semidefinite matrix
$\mathbf{M}$, there exists a positive-semidefinite matrix $\mathbf{N}$ of rank
$r$ satisfying $\mathbf{M}_{\,i,i} = \mathbf{N}_{i,i}$ for all
$1 \leqslant i \leqslant n+1$ and $\mathbf{M}_{\,i,j} = \mathbf{N}_{i,j}$ for
each edge $\{i,j\}$ in the graph $G$. Determining the Pythagoras number of the
subvariety $X_G \subseteq \mathbb{P}^n$, defined as zero-locus of the
quadratic monomials $x_i \, x_{\!j}$ for every pair $\{ i, j \}$ of distinct
vertices that do not form an edge in the graph $G$, is the algebro-geometric
reformulation.  Example~\ref{e:cycle} and Example~\ref{e:Peter} specialize the
bounds in Theorem~\ref{t:1} for cycles and the Petersen graph respectively.
Providing even further evidence of the broad scope of the geometric approach,
Corollary~\ref{c:tree} rediscovers Lemma~2.7 in \cite{LV} and
Remark~\ref{r:tree} raises an enticing analogy between the treewidth of a
graph and the dimension of a $2$\nobreakdash-regular variety.

%% ---------------------------------------------------------------------------
\subsection*{Organization}

Section~\ref{s:py} proves all three of the upper bounds in
Theorem~\ref{t:1}. Examples show that these upper bounds on the Pythagoras
number of a real subvariety can be sharp and are frequently inequivalent.  In
Section~\ref{s:qp}, we introduce quadratic persistence, outline the essential
properties of this new numerical invariant, and derive the lower bounds
appearing in Theorem~\ref{t:3}.  Section~\ref{s:res} relates the quadratic
persistence of a complex subvariety to the linear syzygies of its homogeneous
ideal via the Bernstein--Gelfand--Gelfand correspondence.  Finally,
Section~\ref{s:toric} hones our bounds on the quadratic persistence for some
projective toric subvarieties.

%% ---------------------------------------------------------------------------
\subsection*{Conventions}

Throughout the article, the set of nonnegative integers denoted by
$\mathbb{N}$.  For any $n \in \mathbb{N}$, let
$S \coloneqq \mathbb{R}[x_0, x_1, \dotsc, x_n]$ be the polynomial ring with
the standard $\mathbb{N}$\nobreakdash-grading induced by setting
$\deg(x_i) = 1$ for all $0 \leqslant i \leqslant n$.  A real quadratic
function $f$ is positive semidefinite if it is nonnegative on
$\mathbb{R}^{n+1}$.  We write $\mathscr{S}_+$ for the closed convex cone of
positive-semidefinite forms in $S_{2}$.

A real projective subvariety
$X \subseteq \mathbb{P}^n \coloneqq \operatorname{Proj}(S)$ is a reduced
subscheme of projective space over the field $\mathbb{R}$ of real numbers.
Likewise, a complex projective subvariety
$X \subseteq \mathbb{P}^n \coloneqq \operatorname{Proj}(\mathbb{C}[x_0, x_1,
\dotsc, x_n])$ is a reduced subscheme of projective space over the field
$\mathbb{C}$ of complex numbers.  We do \emph{not} require a variety to be
irreducible.  A projective subvariety is non-degenerate if it is not contained
in a hyperplane.  The variety $X$ is totally real if the set $X(\mathbb{R})$
of real points is Zariski dense in its set $X(\mathbb{C})$ of complex points
or, equivalently, if every irreducible component of $X$ has a nonsingular real
point.

%%%%%%%%%%%%%%%%%%%%%%%%%%%%%%%%%%%%%%%%%%%%%%%%%%%%%%%%%%%%%%%%%%%%%%%%%%%%%%
\section{Upper bounds on the Pythagoras number}
\label{s:py}

\noindent
We provide three different upper bounds on the Pythagoras number of a
projective subvariety.  In addition, we specialize these results to obtain
concrete upper bounds on the Pythagoras numbers of several classes including
varieties of small degree, projective curves, varieties associated to graphs,
and toric subvarieties.  We also show that these bounds are often sharp and,
in general, incomparable.

%% ---------------------------------------------------------------------------
\subsection*{An upper bound through convex geometry}

To establish our first bound, we exploit the existence of low-rank matrices on
sufficiently large affine subspaces of quadratic forms that intersect the cone
$\mathscr{S}_{+}$ of positive-semidefinite forms.  Given a real subvariety
$X \subseteq \mathbb{P}^n$, we write $I_X$ for its saturated homogeneous ideal
in $S \coloneqq \mathbb{R}[x_0, x_1, \dotsc, x_n]$ and $R \coloneqq S/I_X$
denotes its homogeneous coordinate ring.  Let
$\Sos_X \coloneqq \left\{ f \in R_2 \; \middle| \; \text{there exist
    $g_1, g_2, \dotsc, g_r \in R_{1}$ such that
    $f = g_1^2 + g_2^2 + \dotsb + g_r^2$} \right\}$ be the convex cone of sums
of squares in $R_2$ and define the Pythagoras number of the variety $X$ to be
\[
  \py(X) \coloneqq \min \left\{ r \in \mathbb{N} \; \middle| \;
    \text{
      \begin{minipage}[c]{210pt}
        for all $f \in \Sos_X$, there
        exists $g_1, g_2, \dotsc, g_r \in R_{1}$ such that
        $f = g_1^2 + g_2^2 + \dotsb + g_r^2$
      \end{minipage}
    } \right\} \, .
\]
We first strengthen the non-strict inequality derived from Theorem~4.4 and
Corollary~5.3 in \cite{CLR}.

\begin{theorem} 
  \label{t:convex}
  Let $X \subseteq \mathbb{P}^n$ be a real subvariety such that
  $X(\mathbb{R})$ is non-degenerate.  When the ideal $I_X$ contains at least
  one nonzero quadratic form, we have the inequality
  $\binom{\py(X)+1}{2} < \dim_{\mathbb{R}} (R_2)$.
\end{theorem}

\begin{proof} 
  Set $I \coloneqq I_X$, let $\eta_{2} \colon S_2 \to R_2 = S_2/I_2$ denote
  the degree-two piece of the canonical quotient map, and set
  $r \coloneqq \py(X) - 1$.  Fix a nonzero $f \in \Sos_X$.  Since $X$ is
  non-degenerate, there exists a polynomial representative $f \in S_2$ such
  that $\eta_{2}(\tilde{f}) = f$ and $\tilde{f}$ is also a sum of squares.
  Each nonzero quadratic form in $S$ corresponds to a real symmetric
  $(n+1) \times (n+1)$-matrix and this form is a sum of squares in $S$ if and
  only if the matrix is positive semidefinite.  Better yet, the form is a sum
  of $r$ squares if and only if the corresponding positive-semidefinite matrix
  has rank $r$.  Let $A$ be the affine subspace of symmetric matrices
  corresponding to $\tilde{f} + I_2$, that is all polynomial representatives
  of $f$. By construction, the affine subspace $A$ has dimension equal to
  $\dim_{\mathbb{R}} (I_{2})$ and codimension equal to
  $\dim_{\mathbb{R}} (R_2)$.  Moreover, $A$ has a nonempty intersection with
  the cone $\mathscr{S}_{+}$ of positive-semidefinite matrices because
  $f \in \Sos_X$.  Since $X(\mathbb{R})$ is non-degenerate, the vector space
  $I_2$ does not contain a sum of squares and the intersection
  $A \cap \mathscr{S}_+$ is compact.  If
  $\codim(A) = \dim_{\mathbb{R}} (R_2) < \binom{\py(X)+1}{2} =
  \binom{r+2}{2}$, then Proposition~II.13.1 in \cite{Bar} implies that there
  exists a matrix in $A \cap \mathscr{S}_{+}$ with rank at most
  $r = \py(X) - 1$.  However, this would contradict the definition of the
  Pythagoras number, so we deduce that
  $\codim(A) = \dim_{\mathbb{R}} (R_2) \geqslant \binom{\py(X)+1}{2} =
  \binom{r+2}{2}$.

  It remains to prove that the equality
  $\codim(A) = \dim_{\mathbb{R}} (R_2) = \binom{r+2}{2} = \binom{\py(X)+1}{2}$
  is also impossible.  If $r = 0$ and $\dim_{\mathbb{R}} R_2 = 1$, then the
  Macaulay Characterization Theorem~\cite{HH}*{Theorem~6.3.8} shows that the
  Hilbert function of $X$ equals $1$ for all integers greater than $1$, so $X$
  is a single point and, hence, degenerate.  Finally, suppose that $r > 0$ and
  $\dim_{\mathbb{R}} R_2 = \binom{r+2}{2}$.  Since every quadratic form in
  $S_{2}$ has rank at most $n+1$, we see that $\py(X) \leqslant n+1$.
  However, the ideal $I$ contains, by hypothesis, at least one nonzero
  quadratic form, so it follows that $\py(X) < n+1$ and $r+2 \leqslant n+1$.
  Thus, Proposition~II.13.4 in \cite{Bar} proves that there is a matrix in
  $A \cap \mathscr{S}_+$ with rank at most $r = \py(X) - 1$ which again
  contradicts the definition of the Pythagoras number.
\end{proof}

%% ---------------------------------------------------------------------------
\subsection*{An upper bound from differential topology}

To prove our second bound, we rely on a topological argument originating in
Hilbert's proof \cite{Hil} that every nonnegative ternary quartic is a sum of
$3$ squares.  More recently, Theorem~3.5 in \cite{CPSV} and Section~2 in
\cite{BPSV} develop variants.  Our version depends on a technical property of
a basepoint-free linear series; compare with the
$p$\nobreakdash-base-point-free property in Subsection~1.2 of
\cite{BSV17}. Following Definition~6.0.23 in \cite{CLS}, a linear series
$W \subseteq R_1$ is basepoint-free if the linear forms in $W$ have no common
zeroes (neither real nor complex) on the underlying variety $X$.

\begin{theorem}
  \label{t:series} 
  Let $X \subseteq \mathbb{P}^n$ be a real subvariety such that
  $X(\mathbb{R})$ is non-degenerate.  If $k \in \mathbb{N}$ is the smallest
  integer such that any basepoint-free linear series $W \subseteq R_1$ of
  dimension $k$ generates all of $R_2$, then we have $\py(X) \leqslant k$.
\end{theorem}

\begin{proof}
  Any linear series of dimension at most $\dim(X)$ determines a nonempty
  subscheme of $X$, so we may assume that $k > \dim(X)$ and a general linear
  series in $R_1$ of dimension $k$ is basepoint-free.  For any positive
  integer $r$, let $\varsigma_{\,r} \colon \bigoplus_{i=1}^r R_1 \to R_2$ be
  the map defined by $\varsigma(g_1, g_2, \dotsc, g_r) = \sum_{i=1}^r g_i^2$.
  It suffices to prove that $\operatorname{Im}(\varsigma_{\,k}) = \Sos_X$.

  We begin with a connectedness observation.  Since $X(\mathbb{R})$ is
  non-degenerate, we may regard $\varsigma_r$ as a continuous map from
  $\mathbb{P}(\bigoplus_{i=1}^r R_1)$ to $\mathbb{P}(R_2)$ where both spaces
  are endowed with the Euclidean topology as in Lemma~2.2 of \cite{BPSV}.  As
  a continuous map between compact Hausdorff spaces, it is both proper and
  closed.  The differential $\operatorname{d}\!\varsigma_r$ at the point
  $(g_1, g_2, \dotsc, g_r)$ sends the $r$\nobreakdash-tuple of linear forms
  $(h_1, h_2, \dotsc, h_r)$ to the sum $2 \sum_{i=1}^r h_i g_i$, so the image
  is the graded component of the ideal generated by linear forms, namely the
  $\mathbb{R}$\nobreakdash-vector space $\ideal{g_1, g_2, \dotsc, g_r}_2$.
  The defining condition for $k$ implies that the differential
  $\operatorname{d}\!\varsigma_{\, k}$ is surjective at all points
  $(g_1, g_2, \dotsc, g_k)$ where the homogeneous polynomials
  $g_1, g_2, \dotsc, g_k$ are linearly independent and do not have a common
  zero on $X$.  If $\Lambda$ denotes the branch locus of $\varsigma_{\, k}$
  and $\Delta$ is the Zariski closure of all quadratic forms that are singular
  at a smooth point of $X$ (also known as the discriminant variety), then
  $\Phi \coloneqq \Sos_X \setminus (\Lambda \cup \Delta)$ is a dense subset of
  $\Sos_X$ in the Euclidean topology.  The implicit function theorem shows
  that the subset $\operatorname{Im}(\varsigma_{\, k}) \cap \Phi$ is open.
  The subset $\operatorname{Im}(\varsigma_{\, k}) \cap \Phi$ is also closed
  (in $\Phi$) because the map $\varsigma_{\, k}$ is closed and the hypothesis
  that $k > \dim(X)$ ensures that it is nonempty.  Thus, the intersection
  $\operatorname{Im}(\varsigma_{\, k}) \cap \Phi$ is a union of connected
  components of $\Phi$.

  Using this connectivity, we complete the proof.  A real quadratic form lies
  in $\Delta \cap \operatorname{int}(\Sos_X)$ if and only if there exists a
  conjugate pair of complex points in $X$ at which it is singular, so the set
  $\Delta \cap \operatorname{int}(\Sos_X)$ has codimension at least $2$ in
  $R_2$.  Since $\dim \Sos_X = \dim_{\mathbb{R}} R_2$, we see that
  $\Sos_X \setminus \Delta$ is connected.  If $\Phi$ is also connected, as
  occurs when the branch locus $\Lambda$ is empty, then we have
  $\Phi \subseteq \operatorname{Im}(\varsigma_{\, k})$. If not, then $\Lambda$
  is a divisor and two connected components of $\Phi$ are separated by an
  irreducible component $Z$ of the branch locus $\Lambda$. In particular,
  there is a real smooth point $z$ on the hypersurface $Z$ lying in
  $\operatorname{int}(\Sos_X) \setminus \Delta$.  Since
  $z \in \Lambda \setminus \Delta$, there exists
  $g_1, g_2, \dotsc, g_k \in R_1$ having no common zero in $X$ such that
  $z = g_1^2 + g_2^2 + \dotsb + g_k^2$, but the
  $\mathbb{R}$\nobreakdash-vector space
  $\ideal{g_1, g_2, \dotsc, g_k}_2 \subset R_2$ has codimension $1$
  contradicting the defining condition for $k$.  It follows that
  $\Lambda \subset \Delta$, which implies that $\Phi$ is connected.  Since
  $\Phi$ is dense in $\Sos_X$ and $\varsigma_{\, k}$ is closed, we conclude
  that $\operatorname{Im}(\varsigma_{\, k}) = \Sos_X$.
\end{proof}

The bound in Theorem~\ref{t:series} is hard to determine precisely.
Nonetheless, it is related to the \define{Green--Lazarsfeld index}, which is
defined to be
\begin{align*}
  \green(X) \coloneqq \max \left\{ j \in \mathbb{N} \; \middle| \;
  \text{$\smash{\Tor_k^{\, S}(R, \mathbb{R})_{2+k}} = 0$ for all $k \leqslant j$}
  \right\} \, .
\end{align*}
In other words, $\green(X)$ is the largest $k \in \mathbb{N}$ such that the
homogeneous ideal $I_X$ is generated by quadrics and the first $k-1$ maps in
its minimal free resolution are represented by matrices of linear forms; see
Remark~\ref{r:Betti} and page~155 in \cite{Eis}.

\begin{corollary} 
  \label{c:gl}
  For any real subvariety $X \subseteq \mathbb{P}^n$ such that $X(\mathbb{R})$
  is non-degenerate, we have 
  \[
    \py(X) \leqslant n + 1 - \min\left\{ \green(X), \codim(X) \right\} \, .
  \]
\end{corollary}

\begin{proof}
  Theorem~6 in \cite{BSV17} shows that any basepoint-free linear series of
  dimension $n+1-k$ generates all of $R_2$ when the homogeneous ideal $I_X$ is
  generated by quadrics and its first $k-1$ syzygies are linear.  Thus, the
  assertion follows immediately from Theorem~\ref{t:series}.
\end{proof}

We also recover Theorem~3.5 in \cite{CPSV}.

\begin{corollary}
  \label{c:acmUp}
  Let $X \subseteq \mathbb{P}^n$ be an irreducible real subvariety such that
  $X(\mathbb{R})$ is non-degenerate.  If $X$ is arithmetically Cohen--Macaulay
  and $\deg X = 2 + \codim X$, then we have $\py(X) \leqslant 2 + \dim(X)$.
\end{corollary} 

\begin{proof} 
  When $X$ is hypersurface, the statement is trivial. If $X$ is not a
  hypersurface, then Theorem~4.3 in \cite{HK12} shows that $X$ is an
  arithmetically Cohen--Macaulay variety such that $\deg X = 2 + \codim X$ if
  and only if $\green(X) = \codim(X) - 1$.  Hence, Corollary~\ref{c:gl}
  establishes that $\py(X) \leqslant 2 + \dim(X)$.
\end{proof}

%% ----------------------------------------------------------------------------
\subsection*{Upper bounds via embeddings}

Our third bound comes from embeddings into a special type of variety.  We
start with an elementary inequality among Pythagoras numbers.

\begin{lemma} 
  \label{l:inclPy}
  An inclusion of real subvarieties $X \subseteq X' \subseteq \mathbb{P}^n$
  produces the inequality $\py(X) \leqslant \py(X')$.
\end{lemma}

\begin{proof}
  Let $R' \coloneqq \mathbb{R}[x_0, x_1, \dotsc, x_n]/I_{X'}$ denote the
  homogenenous coordinate ring of $X'$ in $\mathbb{P}^n$.  The inclusion
  $X \subseteq X'$ corresponds to an $\mathbb{N}$\nobreakdash-graded
  surjective ring homomorphism $\varphi \colon R' \to R$, so every square in
  $R$ is the image of a square in $R'$ and $\Sos_{X} = \varphi(\Sos_{X'})$.
  If an element $f \in \Sos_X$ satisfies $f = \varphi(f')$ for some
  $f' \in \Sos_{X'}$ and $f'$ can be expressed as a sum of $k$ squares, then
  we obtain an expression for $g$ involving at most $k$ squares by applying
  $\varphi$.  It follows that $\py(X) \leqslant \py(X')$.
\end{proof}

To capitalize on this lemma, we need to know the Pythagoras numbers for a
class of subvarieties. With this in mind, a subvariety
$X' \subseteq \mathbb{P}^n$ is \define{$2$\nobreakdash-regular} (in the sense
of Castelnuovo--Mumford) if its homogeneous ideal $I_{X'}$ is generated by
quadratic polynomials and all the maps in its minimal free resolution are
represented by matrices of linear forms or, equivalently,
$\green(X') = \infty$; see Section~4A in \cite{Eis}.  Fortuitously,
Corollary~32 in \cite{BSV17} shows that the Pythagoras number for any
totally-real $2$\nobreakdash-regular subvariety $X'$ is $1 + \dim(X')$.
Motivated by this, our third bound revolves around embeddings into
$2$\nobreakdash-regular subvarieties.

\begin{theorem}
  \label{t:2reg} 
  For any real subvariety $X \subseteq \mathbb{P}^n$ such that $X(\mathbb{R})$
  is non-degenerate, its Pythagoras number $\py(X)$ is at most one more than
  the minimum dimension of any real $2$\nobreakdash-regular variety that
  contains it.
\end{theorem}

\begin{proof}
  Let $X'$ be a real $2$\nobreakdash-regular variety such that
  $X \subseteq X' \subseteq \mathbb{P}^n$.  Since $\green(X') = \infty$,
  Corollary~\ref{c:gl} gives
  $\py(X') \leqslant n + 1 - \codim X' = 1 + \dim(X')$ and
  Lemma~\ref{l:inclPy} completes the proof.
\end{proof}

\begin{proof}[Proof of Theorem~\ref{t:1}]
  Theorem~\ref{t:convex} proves the first part,
  Corollary~\ref{c:gl} proves the second, and
  Theorem~\ref{t:2reg} proves the third.
\end{proof}

For irreducible subvarieties, we can improve this bound.  A projective
subvariety $X' \subset \mathbb{P}^n$ has \define{minimal degree} if it is
non-degenerate and $\deg X' = 1 + \codim X'$. Theorem~0.4 in \cite{EGHP6}
gives a complete classification of $2$\nobreakdash-regular varieties: the
irreducible components are varieties of minimal degree that meet in a
particularly simple way.  Therefore, to bound the Pythagoras number of an
irreducible subvariety, one need only consider the varieties of minimal degree
that contain it.  Moreover, the Del~Pezzo--Bertini
Theorem~\cite{EH}*{Theorem~1} proves that an irreducible variety of minimal
degree is either a quadric hypersurface, a rational normal scroll, or a cone
over the Veronese surface $\mathbb{P}^2 \subset \mathbb{P}^5$.  Concentrating
on just the rational normal scrolls that contain an irreducible variety
produces the next bound.  As in Section~6C in \cite{Eis}, a projective
subvariety $X \subseteq \mathbb{P}^n$ is \define{linearly normal} if the
canonical map
$H^0 \bigl( \mathbb{P}^n, \mathcal{O}_{\mathbb{P}^n}(1) \bigr) \to H^0 \bigl(
X, \mathcal{O}_{X}(1) \bigr)$ is surjective

\begin{corollary}
  \label{c:scroll} 
  Let $X \subseteq \mathbb{P}^n$ be a non-degenerate irreducible real
  subvariety.  If $X$ is linearly normal, then the Pythagoras number $\py(X)$
  is at most
  \[
    \min \left\{ n + 2 - \dim_{\mathbb{R}} H^0 \bigl(X, \mathcal{O}_X(1) \otimes
      \mathcal{L}^{-1} \bigr) \; \middle| \; \text{
        \begin{minipage}[c]{5.1cm}
          \raggedright
          $\mathcal{L}$ is a real line bundle on $X$ such that
          $\dim_{\mathbb{R}} H^0(X, \mathcal{L}) \geqslant 2$
        \end{minipage}
      } \right\} \, .
  \]
\end{corollary}

\begin{proof}
  In light of Theorem~\ref{t:2reg}, it suffices to prove that the dimension of
  a rational normal scroll containing $X$ equals
  $n + 1 - \dim_{\mathbb{R}} H^0 \bigl(X, \mathcal{O}_X(1) \otimes
  \mathcal{L}^{-1} \bigr)$ for some real line bundle $\mathcal{L}$ on $X$
  satisfying $\dim_{\mathbb{R}} H^0(X, \mathcal{L}) \geqslant 2$.
  Paragraph~2.2 in \cite{Sch86} indicates that we can construct from any
  pencil of divisors in $\abs{\mathcal{L}}$ on $X$ satisfying
  $\dim_{\mathbb{R}} H^0(X, \mathcal{O}_X(1) \otimes \mathcal{L}^{-1} \bigr)
  \geqslant 2$ a rational normal scroll $X' \subseteq \mathbb{P}^{n}$ which
  contains $X$.
  % Suppose
  % that $X' \subseteq \mathbb{P}^n$ is a rational normal scroll containing $X$.
  % A rational normal scroll is a cone over a smooth linearly normal variety
  % fibered over $\mathbb{P}^1$ by linear spaces, so there is a morphism
  % $\varpi \colon X' \to \mathbb{P}^1$ whose fibers are linear.  The
  % restriction $\varpi|_X \colon X \to \mathbb{P}^1$ is surjective because $X$
  % is non-degenerate.  Hence, the fibers of $\varpi|_X$ form a pencil (or
  % $1$-dimensional family) of linearly equivalent divisors on $X$.  Let
  % $\mathcal{L}$ be a real line bundle on $X$ such that this pencil corresponds
  % to a subvector space of $H^0(X, \mathcal{L})$.  The non-degeneracy of $X$
  % also implies that $\dim_{\mathbb{R}} H^0(X, \mathcal{L}) \geqslant 2$.
  % Theorem~2 in \cite{EH} establishes that the variety formed by the union of
  % the linear spans of the divisors in this pencil is a rational normal scroll.
  % The linear span of a divisor in the pencil is the intersection of all the
  % hyperplanes containing it, which corresponds to an element of
  % $H^0\bigl(X, \mathcal{O}_X(1) \otimes \mathcal{L}^{-1} \bigr)$.  Hence, the
  % fiber of the morphism $\varpi$ is the linear span of the corresponding
  % divisor in the pencil and we have reconstructed the rational normal scroll
  % $X'$.
  Since $X'$ is variety of minimal degree and
  $\deg(X') = \dim_{\mathbb{R}} H^0 \bigl( X, \mathcal{O}_X(1) \otimes
  \mathcal{L}^{-1} \bigr)$, it follows that
  $\dim(X') = n + 1 - \dim_{\mathbb{R}} H^0\bigl(X, \mathcal{O}_X(1) \otimes
  \mathcal{L}^{-1} \bigr)$.
\end{proof}

To illustrate this corollary, we specialize to curves whose hyperplane section
is non-special.  Emulating the definition in Section~8C in \cite{Eis}, the
\define{real gonality} of a real curve is the lowest degree of a real
non-constant morphism from the curve to the real projective line. In
particular, the real gonality of a real curve $X \subseteq \mathbb{P}^n$ is at
least the gonality of its complexification
$X \times_{\operatorname{Spec}(\mathbb{R})} \operatorname{Spec}(\mathbb{C})$.

\begin{corollary}
  \label{c:gon}
  Let $X \subset \mathbb{P}^n$ be a linearly-normal irreducible non-singular
  real curve of genus $g$ and real gonality $\delta$.  If $X$ has degree at
  least $2g - 1 + \delta$, then we have $\py(X) \leqslant 1 + \delta$.
\end{corollary}

\begin{proof}
  Since the real gonality of $X$ is $\delta$, there is a non-constant morphism
  $\varpi \colon X \to \mathbb{P}^1$ of schemes over $\mathbb{R}$ having
  degree $\delta$.  Fix a real divisor $D$ in the complete linear series of
  the real line bundle $\varpi^* \bigl( \mathcal{O}_{\mathbb{P}^1}(1) \bigr)$.
  We must have $\dim_{\mathbb{R}} H^0 \bigl( X, \mathcal{O}_{X}(D) \bigr) = 2$
  because otherwise there would be a real point $Q \in X$ such that
  $\dim_{\mathbb{R}} H^0 \bigl( X, \mathcal{O}_{X}(D-Q) \bigr) \geqslant 2$
  and the line bundle $\mathcal{O}_X(D-Q)$ would define a real morphism to
  $\mathbb{P}^1$ of smaller degree.  Let $H$ be a hyperplane section of $X$
  and let $K$ be the canonical divisor on $X$.  It follows that
  $\deg (H) \geqslant 2g - 1 + \delta$ and $\deg (K) = 2 g - 2$, so
  $\deg (K-H+D) < 0$ and $\deg (K-H) < 0$.  As
  $\dim_{\mathbb{R}} H^0 \bigl( X, \mathcal{O}_X(K-H+D) \bigr) = 0 =
  \dim_{\mathbb{R}} H^0 \bigl( X, \mathcal{O}_X(K-H) \bigr)$, the
  Riemann--Roch Theorem shows that
  $\dim_{\mathbb{R}} H^0 \bigl( X, \mathcal{O}_X(H-D) \bigr) = \deg(H) -
  \delta + 1 - g$ and
  \[
    n+1 = \dim_{\mathbb{R}} H^0 \bigl( X, \mathcal{O}_X(H) \bigr) = \deg(H) +
    1 - g \, .
  \]
  Therefore, Corollary~\ref{c:scroll} establishes that
  $\py(X) \leqslant n+2 - \dim_{\mathbb{R}} H^0 \bigl( X, \mathcal{O}_X(H-D)
  \bigr) = 1 + \delta$.
\end{proof}

For a canonical curve (a non-hyperelliptic smooth curve of genus $g$ at least
$3$ embedded by its canonical linear series), we get a slightly better bound.

\begin{corollary}
  \label{c:canon}
  If $X \subset \mathbb{P}^{g-1}$ is a canonical real curve of real gonality
  $\delta$, then we have $\py(X) \leqslant \delta$.
\end{corollary}

\begin{proof}
  Just as in the proof of Corollary~\ref{c:gon}, let $D$ denote a real divisor
  on the curve $X$ of degree $\delta$ such that
  $\dim_{\mathbb{R}} H^0 \bigl( X, \mathcal{O}_{X}(D) \bigr) = 2$.  Since the
  canonical divisor $K$ on $X$ corresponds to a hyperplane section, the
  Riemann--Roch Theorem shows that
  \[
    \dim_{\mathbb{R}} H^0 \bigl( X, \mathcal{O}_X(K-D) \bigr) = - \deg(D) -1 +
    g + \dim_{\mathbb{R}} H^0 \bigl( X, \mathcal{O}_X(D) \bigr) = g + 1 -
    \delta \, .
  \]
  Thus, Corollary~\ref{c:scroll} demonstrates that
  $\py(X) \leqslant (g - 1) + 2 - (g + 1 - \delta) = \delta$.
\end{proof}

For general real canonical curves, we can compare our three bounds.

\begin{example}[Bounds for general canonical curves]
  \label{e:genCan}
  Suppose that $X \subset \mathbb{P}^{g-1}$ is a general real canonical curve
  and let $K$ denote its canonical divisor.  Since $\deg K = 2g-2$, the
  Riemann--Roch Theorem shows that
  $\dim_{\mathbb{R}} H^0 \bigl(X, \mathcal{O}_X(2K) \bigr) = 2(2g-2)+1-g =
  3g-3$, so the first bound derived from Theorem~\ref{t:convex} is
  \[
    \py(X) \leqslant \bigl\lfloor \tfrac{1}{2} \bigl( \sqrt{24g-23} -1 \bigr)
    \bigr\rfloor \leqslant \bigl\lfloor \sqrt{6g} \bigr\rfloor \, .
  \]
  Green's Conjecture~\cite{Eis}*{Conjecture~9.6} asserts that
  $\green(X) \leqslant \left\lfloor \frac{1}{2}(g-1) \right\rfloor - 1$ and it
  is known to hold for general curves~\cite{Voi}.  Thus, the second bound
  obtained from Corollary~\ref{c:gl} is
  \[
    \py(X) \leqslant g - \left\lfloor \tfrac{1}{2}(g-1) \right\rfloor + 1 =
    \left\lfloor \tfrac{1}{2}(g+4) \right\rfloor \, .
  \]
  Lastly, the Brill--Noether Theorem~\cite{Eis}*{Theorem~8.16} implies that
  the (complex) gonality of a general curve is
  $\left\lfloor \frac{1}{2}(g+3) \right\rfloor$, so the third bound from
  Corollary~\ref{c:canon} is at best
  $\py(X) \leqslant \left\lfloor \frac{1}{2}(g+3) \right\rfloor$.  In
  particular, for all sufficiently large $g$, the first bound is stronger than
  the other two bounds.
\end{example}

%% ----------------------------------------------------------------------------
\subsection*{Specific bounds for graphs}

Restricting our attention to certain unions of coordinate spaces allows us to
compare our three bounds on the Pythagoras number.  We focus on varieties
defined by the Stanley--Reisner ideal of the clique complex of a graph or,
equivalently, the edge ideal of the complementary graph.  Remarkably, all
three bounds have explicit formulations in terms of well-known numerical
invariants of the underlying graph.

To be more precise, let $G$ be a graph (with no multiple edges or loops) whose
vertex set is $\{0, 1, \dotsc, n\}$.  The homogeneous ideal $I_G$ in
$S \coloneqq \mathbb{R}[x_0, x_1, \dotsc, x_n]$ is generated by the quadratic
monomials $x_i \, x_{\!j}$ for every pair $\{i, j \}$ of distinct vertices that
do not form an edge in the graph $G$; see Section~9.1 in \cite{HH}. The
associated subvariety is
$X_G \coloneqq \operatorname{V}(I_{G}) \subseteq \mathbb{P}^n$ and its
homogeneous coordinate ring is
$R_G \coloneqq \mathbb{R}[x_0, x_1, \dotsc, x_n]/I_G$.  If the graph $G$ also
has $m$ edges, then the definition of the ideal $I_G$ implies that
$\dim_{\mathbb{R}} (R_G)_2 = \tbinom{n+2}{2} - \tbinom{n+1}{2} + m = n + m +
1$.  Hence, the first bound derived from Theorem~\ref{t:convex} is
\[
  \py(X_G) \leqslant \bigl\lfloor \tfrac{1}{2} \bigl( \sqrt{8n+8m+9} -1 \bigr)
  \bigr\rfloor \, .
\]
Using `Gram dimension' as a synonym for the Pythagoras number, this bound also
appears in the introduction to \cite{LV}.

For the second bound, we translate both the Green--Lazarsfeld index and the
dimension of $X_G$ into numerical graph invariants.  To do this for the index
$\green(X_G)$, recall that a \define{cycle} in the graph $G$ of length
$m \geqslant 3$, as defined in Section~1.3 of \cite{Die}, is determined by a
sequence of distinct vertices $v_0, v_1, \dotsc, v_{m-1}$ such that each of
the pairs
$\{v_0,v_1\}, \{v_1,v_2\}, \dotsc, \{v_{m-2}, v_{m-1}\}, \{v_{m-1}, v_{0}\}$
is an edge in the graph.  An edge that joins two vertices of a cycle but is
not itself an edge of the cycle is called a chord and an \define{induced
  cycle} has no chords.  Theorem~2.1 in \cite{EGHP5} proves that the
Green--Lazarsfeld index $\green(X_G)$ is $3$ less than minimal length of an
induced cycle in $G$ having length at least $4$.  To reinterpret $\dim(X_G)$,
recall that a clique in the graph $G$ is a subset of vertices such that every
pair of distinct vertices forms an edge and the \define{clique number}
$\omega(G)$ is the number of vertices in a maximum clique; see Section~5.5 in
\cite{Die}.  Lemma~1.5.4 in \cite{HH} shows that the primary decomposition of
$I_{G}$ is the intersection of monomial prime ideals generated by the
variables corresponding to the complement of a maximum clique, so
$\dim(X_G) = \omega(G) - 1$.  Thus, if $\iota(G)$ is the minimal length of an
induced cycle in $G$ having length at least $4$, then Corollary~\ref{c:gl}
gives
\[
  \py(X_G) \leqslant \max\{ n - \iota(G) + 4, \omega(G) \} \, .
\]

The third type of bound depends on a more subtle numerical invariant of $G$.
A graph is \define{chordal} if every induced cycle has exactly three vertices;
again see Section~5.5 in \cite{Die}. Proposition~12.4.4 in \cite{Die}
demonstrates that the \define{treewidth} of $G$ is one less than the size of
the largest clique in a chordal graph containing $G$ with the smallest clique
number.  In this setting, the explicit form of the third bound rediscovers
Lemma~2.7 in \cite{LV}.

\begin{corollary}
  \label{c:tree}
  For any graph $G$, the Pythagoras number $\py(X_G)$ of its associated
  subvariety is at most one more than the treewidth of the underlying graph
  $G$.
\end{corollary}

\begin{proof}
  The definition of Stanley--Reisner ideals implies that one has an
  containment of graphs $G \subseteq G'$ if and only if one has a containment
  of varieties $X_{G} \subseteq X_{G'}$.  The Fr\"{o}berg
  Theorem~\cite{HH}*{Theorem~9.2.3} asserts that the Stanley--Reisner ideal
  $I_{G'}$ is $2$\nobreakdash-regular if and only if the graph $G'$ is
  chordal.  Therefore, the minimum dimension of any real
  $2$\nobreakdash-regular variety containing $X_G$ is at most the treewidth of
  $G$ and appealing to Theorem~\ref{t:2reg} finishes the proof.
\end{proof}

\begin{remark}
  \label{r:tree}
  Corollary~\ref{c:tree} demonstrates that the upper bound in
  Theorem~\ref{t:2reg} specializes to the treewidth of a graph. For a
  non-degenerate subvariety $X \subseteq \mathbb{P}^n$, to what extent is this
  numerical invariant, namely one more that the minimum dimension of any
  $2$\nobreakdash-regular variety in $\mathbb{P}^n$ that contains $X$, the
  natural geometric generalization of treewidth?
\end{remark}

We contrast our three bounds on Pythagoras numbers for some specific graphs.

\begin{example}[Bounds for cycles]
  \label{e:cycle}
  For any integer $n \geqslant 3$, suppose that the graph $G$ is a cycle on
  $n+1$ vertices.  Since $G$ also has $n+1$ edges, the first bound is
  $\py(X_G) \leqslant \tfrac{1}{2} ( \sqrt{16n+17} -1)$.  The minimum length
  of an induced cycle is $n+1$ and $\omega(G) = 2$, so the second bound
  becomes $\py(X_G) \leqslant 3$.  Lastly, adjoining all the chords incident
  to a fixed vertex yields a chordal graph containing the cycle, so the
  treewidth of $G$ is at most $2$ and the third bound also is
  $\py(X_G) \leqslant 3$.  In particular, the first bound is weaker than the
  other two bounds.
\end{example}

\begin{example}[Bounds for the Petersen graph]
  \label{e:Peter}
  Suppose that the graph $G$ is the Petersen graph; see Figure~6.6.1 in
  \cite{Die}.  Since $G$ has $10$ vertices and $15$ edges, the first bound is
  $\py(X_G) \leqslant 6$.  The minimum length of an induced cycle is $5$ and
  $\omega(G) = 2$, so the second bound becomes $\py(X_G) \leqslant 8$.
  Lastly, the treewidth of the Petersen graph is known to be $4$ (for example
  see Section~3 in \cite{HW}), so the third bound is $\py(X_G) \leqslant 5$.
  Here the third bound is stronger than the other two bounds.
\end{example}

%% ----------------------------------------------------------------------------
\subsection*{Specific bounds for toric subvarieties}

By concentrating on projective toric subvarieties, we can relate our bounds to
the numerical invariants of a lattice polytope.  Consider a lattice polytope
$P \subset \mathbb{R}^d$ containing $n+1$ lattice points, so
$n \coloneqq \abs{P \cap \mathbb{Z}^d} - 1$.  Influenced by Definition~2.1.1
in \cite{CLS}, the associated toric subvariety
$X_{P \cap \smash{\mathbb{Z}^d}} \subseteq \mathbb{P}^n$ is the Zariski
closure of the image of the map from the $d$\nobreakdash-dimensional algebraic
torus to $\mathbb{P}^n$ given by
\[
  (t_1, t_2, \dotsc, t_d) \mapsto \bigl[ t_1^{a_1} \, t_2^{a_2} \, \dotsb \,
  t_d^{a_d} \; \bigl| \; (a_1, a_2, \dotsc, a_d) \in P \cap \mathbb{Z}^d
  \bigr] \, .
\]
We caution that the variety $X_{P \cap \smash{\mathbb{Z}^d}}$ may not be
normal; see Definition~2.3.14 in \cite{CLS} for the canonical normal toric
variety associated to $P$.  Regardless, if
$R \coloneqq \mathbb{R}[x_0,x_2, \dotsc, x_n]/I_{X_{P \cap
    \smash{\mathbb{Z}^d}}}$ is the homogeneous coordinate ring of the
subvariety $X_{P \cap \smash{\mathbb{Z}^d}} \subseteq \mathbb{P}^n$, then
$m \coloneqq \dim_{\mathbb{R}}(R_2)$ equals the number of points in the
Minkowski sum $(P \cap \mathbb{Z}^d) + (P \cap \mathbb{Z}^d)$; compare with
Theorem~1.1.17 in \cite{CLS}.  Hence, the first bound derived from
Theorem~\ref{t:convex} is
$\py(X_{P \cap \smash{\mathbb{Z}^d}}) \leqslant \left\lfloor \frac{1}{2}
  \bigl( \sqrt{8m+1} -1 \bigr) \right\rfloor$.

For the second bound, we would need a polyhedral interpretation of the
Green-Lazarsfeld index $\green(X_{P \cap \smash{\mathbb{Z}^d}})$.  Sadly, we
are unaware of even a reasonable conjectural lower bound for an general
projective toric subvariety.  However, for toric surfaces embedded in
projective space, Corollary~2.1 in \cite{Sch04} proves that
$\green(X_{P \cap \smash{\mathbb{Z}^d}})$ is $3$ less than the number of
lattice points on the boundary of the polygon $P$.  Thus, if
$\operatorname{i}(P)$ denotes the number of lattice points on the interior of
the polygon $P$, then Corollary~\ref{c:gl} gives
$\py(X_{P \cap \smash{\mathbb{Z}^d}}) \leqslant \operatorname{i}(P) + 3$.

The third bound depends on estimating the dimension of the smallest rational
normal scroll than contains the subvariety $X_{P \cap \smash{\mathbb{Z}^d}}$.
Once again, this bound can be found analyzing the lattice points.

\begin{corollary}
  \label{c:lines}
  For any lattice polytope $P \subset \mathbb{R}^d$, its projective toric
  subvariety $X_{P \cap \smash{\mathbb{Z}^d}}$ is contained in a rational
  normal scroll whose dimension is equal to the minimal number of parallel
  lines needed to cover all of the lattice points in $P$, so the Pythagoras
  number $\py(X_{P \cap \smash{\mathbb{Z}^d}})$ is at most one more than the
  dimension of this rational normal scroll.
\end{corollary}

\begin{proof}
  Applying Theorem~\ref{t:2reg}, it suffices to find a rational normal scroll
  containing $X_{P \cap \smash{\mathbb{Z}^d}}$ whose dimension is equal to the
  minimal number of parallel lines needed to cover all of the lattice points
  in $P$.  Suppose that the lattice points in $P$ are covered by $k$ lines
  parallel to the vector $\mathbf{v} \in \mathbb{R}^d$. We may assume that
  $\mathbf{v}$ is a primitive lattice vector.  For each index
  $0 \leqslant i < k$, let $a_i$ be the lattice length of the corresponding
  line segment covering lattice points in $P$.  By relabeling the lines, we
  may also assume that
  $a_{k-1} \geqslant a_{k-2} \geqslant \dotsb \geqslant a_{0} \geqslant 0$.
  Let $\mathbf{e}_1, \mathbf{e}_2, \dotsc, \mathbf{e}_k$ denote the standard
  basis for $\mathbb{Z}^k$ and consider the lattice polytope
  \[
    P' \coloneqq \conv \left\{ \mathbf{0}, \mathbf{e}_1, \mathbf{e}_2, \dotsc,
      \mathbf{e}_{k-1}, a_0 \, \mathbf{e}_k, \mathbf{e}_1 + a_1 \,
      \mathbf{e}_k, \mathbf{e}_2 + a_2 \, \mathbf{e}_k, \dotsc,
      \mathbf{e}_{k-1} + a_{k-1} \, \mathbf{e}_k \right\} \subset \mathbb{R}^k
    \, .
  \]
  By construction, the Lawrence prism $P'$ is the normal full-dimensional
  lattice polytope of a rational normal scroll.  The affine map, which sends
  $\mathbf{e}_k$ to $\mathbf{v}$ and the lattice points
  $\mathbf{0}, \mathbf{e}_1, \mathbf{e}_2, \dotsc, \mathbf{e}_{k-1}$ to the
  minimal points in $P$ relative to the vector $\mathbf{v}$ on the
  corresponding line, defines a bijection between the lattice points in the
  polytopes $P$ and $P'$ and, thereby, induces a toric inclusion
  $X_{P \cap \smash{\mathbb{Z}^d}} \subseteq X_{\smash{P'} \cap
    \smash{\mathbb{Z}^k}}$.
\end{proof}

\begin{remark}
  Since every line bundle on a toric variety is the image of a torus-invariant
  Cartier divisor (see Theorem~4.2.1 in \cite{CLS}), modifying the proof of
  Corollary~\ref{c:scroll} shows that, among all rational normal scrolls
  containing a toric variety, there is one having minimal dimension such that
  the inclusion map is a toric morphism. Hence, the minimal number of parallel
  lines needed to cover all the lattice points in the polytope $P$ is the
  dimension of the smallest rational normal scroll containing the toric
  variety $X_{P \cap \smash{\mathbb{Z}^d}}$.
\end{remark}

\begin{example}[Upper bounds for the Veronese embeddings of $\mathbb{P}^2$]
  \label{e:PP2Up}
  For any integer $j \geqslant 2$, consider the lattice polygon
  $P \coloneqq \conv\{ (0,0), (j,0), (0,j) \} \subset \mathbb{R}^2$. The
  associated toric subvariety $X_{P \cap \smash{\mathbb{Z}^2}}$ is the
  $j$\nobreakdash-th Veronese embedding
  $\mathbb{P}^2 \subset \mathbb{P}^{\binom{j+2}{2}-1}$; see Example~14.2.7 in
  \cite{CLS}.  Since $\dim_{\mathbb{R}}(R_2) = \binom{2j+2}{2}$, the first
  bound is
  \[
    \py(X_{P \cap \smash{\mathbb{Z}^2}}) \leqslant \bigl\lfloor \tfrac{1}{2}
    \bigl( \sqrt{8(j+1)(2j+1) + 1} -1 \bigr) \bigr\rfloor \, .
  \]
  This polygon has $\binom{j-1}{2}$ interior lattice points, so the second
  bound is $\py(X_P) \leqslant \binom{j-1}{2} + 3$.  Lastly, $j+1$ horizontal
  lines cover all the lattice points in $P$, so the third bound is
  $\py(X_{P \cap \smash{\mathbb{Z}^2}}) \leqslant j+2$.  For $j = 2$, the
  second bound is stronger than the other two and is optimal because
  $X_{P \cap \smash{\mathbb{Z}^2}}$ is a variety of minimal degree.  On the
  other hand, the third bound is at least as strong as the other two for all
  $j \geqslant 3$.  For lower bounds on the Pythagoras number of the Veronese
  embeddings of $\mathbb{P}^2$, see Example~\ref{e:PP2Low}.
\end{example}

%%%%%%%%%%%%%%%%%%%%%%%%%%%%%%%%%%%%%%%%%%%%%%%%%%%%%%%%%%%%%%%%%%%%%%%%%%%%%%
\section{Quadratic persistence}
\label{s:qp}

\noindent
This section introduces a numerical invariant of a projective subvariety,
which we call the quadratic persistence.  By definition, this invariant
encodes information about the behaviour of the variety under projections away
from certain linear subspaces.  After summarizing the fundamental features of
this new invariant, we analyze varieties with large quadratic persistence and
find a lower bound on the Pythagoras number of an irreducible totally-real
variety.

%% ---------------------------------------------------------------------------
\subsection*{Properties of quadratic persistence}

Let $X \subseteq \mathbb{P}^n$ be a complex subvariety prescribed by the
saturated homogeneous ideal $I_X$ in polynomial ring
$\mathbb{C}[x_0, x_1, \dotsc, x_n]$.  For any subset
$Z \subseteq \mathbb{P}^n$, the intersection of all linear subspaces of
$\mathbb{P}^n$ which each contain every point in $Z$ is denoted by $\Span(Z)$.
Given a finite set $\Gamma$ of closed points in $X$ spanning a
$(k-1)$\nobreakdash-plane and a complementary linear subspace
$\mathbb{P}^{n-k}$ in $\mathbb{P}^n$, the projection away from $\Gamma$ is the
rational map
$\pi_{\Gamma} \colon \mathbb{P}^n \dashrightarrow \mathbb{P}^{n-k}$ defined by
sending a closed point $q \in \mathbb{P}^n \setminus \Span(\Gamma)$ to the
intersection of $\mathbb{P}^{n-k}$ with the $k$\nobreakdash-plane
$\Span(\{ q \} \cup \Gamma)$. To be notational consistent, we write
$I_{\pi_\Gamma(X)}$ for the saturated homogeneous ideal of the image
$\pi_\Gamma(X) \subseteq \mathbb{P}^{n-k}$. With these preparations, we now
present the key numerical invariant.

\begin{definition}
  \label{d:qp}
  For a complex subvariety $X \subseteq \mathbb{P}^n$, the \define{quadratic
    persistence} $\qp(X)$ is the smallest $k \in \mathbb{N}$ for which there
  exists a finite set $\Gamma$ of closed points in $X$ such that
  $k = \abs{\Gamma}$ and the homogeneous ideal $I_{\pi_\Gamma(X)}$ contains no
  quadratic polynomials.
\end{definition}

The definition of quadratic persistence leads to some easy inequalities.

\begin{lemma}
  \label{l:qpBasic}
  Let $X \subseteq \mathbb{P}^n$ be a complex subvariety.
  \begin{compactenum}[\upshape i.]
  \item If $X$ is non-degenerate, then we have the upper bound
    $\qp(X) \leqslant \codim(X)$.
  \item An inclusion of varieties $X \subseteq X'$ produces the inequality
    $\qp(X') \leqslant \qp(X)$.
  \end{compactenum}
\end{lemma}

\begin{proof} \hfill
  \begin{compactenum}[\itshape i.] 
  \item Fix an irreducible component $Z$ of $X$. Since $X$ is non-degenerate,
    there is a set $\Gamma$ of closed points in $X \setminus Z$ such that
    $\abs{\Gamma} = \codim \bigl( \Span(Z), \mathbb{P}^n \bigr)$ and
    $\Span(\Gamma \cup Z) = \mathbb{P}^n$.  For any set $\Gamma'$ of closed
    points in $Z$ such that $\abs{\Gamma'} = \codim \bigl( Z, \Span(Z) \bigr)$
    and the projection away from $\Gamma'$ is dominant when restricted to $Z$,
    the projection away from $\Gamma \cup \Gamma'$ is also dominant when
    restricted to $Z$ because
    $\codim \bigl( \Span(Z), \mathbb{P}^n \bigr) + \codim \bigl( Z, \Span(Z)
    \bigr) = \codim \bigl( Z, \mathbb{P}^n \bigr)$.  Thus, the ideal
    $I_{\pi_{\Gamma \cup \Gamma'}(X)}$ contains no quadratic polynomials and
    $\qp(X) \leqslant \codim(Z,\mathbb{P}^n)$.
  \item For any finite set $\Gamma \subseteq X \subseteq X'$, we have
    $I_{\pi_\Gamma(X)} \supseteq I_{\pi_\Gamma(X')}$ which gives
    $\qp(X) \geqslant \qp(X')$. \qedhere
  \end{compactenum}
\end{proof}

To better understand quadratic persistence, we examine an auxiliary function
that counts the quadrics kept under an inner projection.  More precisely,
for any finite subset $\Gamma$ of closed points in $X$, set
$\kept_\Gamma(X) \coloneqq \dim (I_{\pi_\Gamma(X)})_2$. Beyond recording the
basic attributes of this function, the following result shows that the
quadratic persistence of an irreducible subvariety is computed by projecting
away from a general set of closed points. Part~\emph{v} appears implicitly in
Theorem~3.1~(a) of \cite{HK15}.

\begin{lemma}
  \label{l:qpGen}
  Let $X \subseteq \mathbb{P}^n$ be a complex subvariety.
  \begin{compactenum}[\upshape i.]
  \item For any finite set $\Gamma$ of closed points in $X$, the number
    $\kept_\Gamma(X)$ is the dimension of the linear subspace spanned by the
    quadrics in $I_X$ that are singular in $\mathbb{P}^n$ at the points of
    $\Gamma$.
  \item An inclusion $\Gamma \subseteq \Gamma'$ of finite subsets of $X$ gives
    the inequality $\kept_\Gamma(X) \geqslant \kept_{\Gamma'}(X)$.
  \item For any $r \in \mathbb{N}$, the function that sends the
    $r$\nobreakdash-tuple $(p_1, p_2, \dotsc, p_r) \in X^r$ of closed points
    to $\kept_{ \{ p_1, p_2, \dotsc, p_r\}}(X)$ is upper semi-continuous.
  \item For any $r \in \mathbb{N}$, the locus in $X^r$ on which the function
    $(p_1, p_2, \dotsc, p_r) \mapsto \kept_{ \{ p_1, p_2, \dotsc, p_r\}}(X)$
    achieves its minimum is Zariski open.
  \item For any finite set $\Gamma$ of closed points in $X$, we have
    $\qp(X) \leqslant \abs{\Gamma} + \qp \bigl( \pi_\Gamma(X) \bigr)$.    
  \item For any closed point $p \in X$, the difference
    $\dim_{\mathbb{C}} (I_X)_2 - \kept_{\{ p \} }(X)$ is the dimension of the
    linear subspace spanned by the gradients of the quadrics in $I_X$
    evaluated at an affine representative of the point $p$.
  \item For any closed point $p \in X$, we have
    $\dim_{\mathbb{C}} (I_X)_2 - \kept_{ \{ p \}}(X) \leqslant \codim X$.
  \end{compactenum}
\end{lemma}

\begin{proof} \hfill
  \begin{compactenum}[\itshape i.]
  \item Choose coordinates $x_0, x_1, \dotsc, x_n$ on $\mathbb{P}^n$ so that
    the linear subspace spanned by $\Gamma$ is cut out by the variables
    $x_{k-1}, x_k, \dotsc, x_{n}$. It follows that
    $I_{\pi_\Gamma(X)} = I_X \cap \mathbb{C}[x_{k-1}, x_{k}, \dotsc, x_n]$;
    see Theorem~8.5.8 in \cite{CLO}. Hence, the graded piece
    $(I_{\pi_\Gamma(X)})_2$ consists of the quadratic polynomials in $I_X$
    that do not involve the variables $x_0, x_1, \dotsc, x_{k-2}$. These are
    precisely the quadrics in $I_X$ that are singular along the linear
    subspace spanned by $\Gamma$ or, equivalently, at the points in $\Gamma$.
  \item Since $\Gamma \subseteq \Gamma'$, the quadrics in $I_X$ singular along
    $\Gamma'$ are contained among those singular along $\Gamma$, so
    part~\emph{i} implies that $\kept_\Gamma(X) \geqslant \kept_{\Gamma'}(X)$.
  \item For any $r \in \mathbb{N}$, consider the incidence correspondence
    $\Psi_r \subseteq \mathbb{P} \bigl( (I_X)_2 \bigr) \times X^r$ consisting
    of all pairs $\bigl( f, (p_1, p_2, \dotsc, p_r) \bigr)$ where the
    quadratic polynomial $f \in I_X$ is singular at all of the closed points
    $p_1, p_2, \dotsc, p_r \in X$. Part~\emph{i} implies that the value of
    $(p_1, p_2, \dotsc, p_r) \mapsto \kept_{ \{ p_1, p_2, \dotsc, p_r \}}(X)$
    is equal to one more than the dimension of the fiber of the projection
    $\operatorname{pr}_2 \colon \Psi_r \rightarrow X^r$. The claim follows
    from the semi-continuity of fibre dimensions; see Th\'eor\`eme~13.1.3 in
    \cite{EGAIV}.
  \item We consider two distinct cases. First, suppose that the image
    $\operatorname{pr}_2(\Psi_r)$ is a proper closed subset of the product
    $X^r$. If $\Gamma$ is a general set of $r$ points on $X$, then the ideal
    $I_{\pi_\Gamma(X)}$ contains no quadratic polynomials and the function
    $(p_1, p_2, \dotsc, p_r) \mapsto \kept_{ \{ p_1, p_2, \dotsc, p_r \}}(X)$
    attains its minimum $0$ on the complement of the image which is a Zariski
    open set. Otherwise, we have $\operatorname{pr}_2(\Psi_r) = X^r$. In this
    case, the minimum of the function
    $(p_1, p_2, \dotsc, p_r) \mapsto \kept_{ \{ p_1, p_2, \dotsc, p_r \}}(X)$
    is some $i \in \mathbb{N}$. This minimum is attained on the complement of
    the sets $\Gamma \subseteq X^r$ of closed points with image greater than
    or equal to $i+1$, which is closed by part~\emph{iii}.
  \item Let $\Gamma' \subseteq \pi_\Gamma(X)$ be a set of closed points such
    that $\abs{\Gamma'} = \qp \bigl( \pi_\Gamma(X) \bigr)$ and the homogeneous
    ideal $I_{\pi_{\smash{\Gamma'}} ( \pi_\Gamma(X) )}$ contains no quadrics.
    Using part~\emph{iv}, we may assume that the subset $\Gamma'$ lies in the
    image of the rational map $\pi_{\Gamma}$.  For each closed point in
    $\Gamma'$, choose a closed point in its fibre contained in $X$, so that
    the resulting finite set
    $\Gamma'' \subseteq \pi_\Gamma^{-1}(\Gamma') \cap X$ has the same
    cardinality as $\Gamma'$ and $\pi_{\Gamma}(\Gamma'') = \Gamma'$. It
    follows that
    $\pi_{\Gamma \cup \Gamma''}(X) = \pi_{\Gamma'} \bigl( \pi_\Gamma(X)
    \bigr)$, so there are no quadratic polynomials in
    $I_{\pi_{\Gamma \cup \Gamma''}(X)}$. Therefore, we conclude that
    $\qp(X) \leqslant \abs{\Gamma \cup \Gamma''} = \abs{\Gamma} +
    \abs{\Gamma'} = \abs{\Gamma} + \qp \bigl( \pi_\Gamma(X) \bigr)$.
  \item Choose an affine representative $\widetilde{p} \in \mathbb{A}^{n+1}$
    of the point $p \in \mathbb{P}^n$ and let
    $\nabla|_{\widetilde{p}} \colon (I_X)_2 \to T_{\mathbb{P}^n,p}^*$ be the
    map defined by sending the quadratic polynomial $f$ to its gradient
    $\nabla f(\widetilde{p})$. Part~\emph{i} implies that the kernel of this
    map is $(I_{\pi_{\{p\}}(X)})_2$, so
    $\dim_{\mathbb{C}} (I_X)_2 - \dim_{\mathbb{C}} (I_{\pi_\Gamma(X)})_2 =
    \operatorname{rank}( \nabla|_{\widetilde{p}} )$.
  \item Since every point in $T_{X,p}$ is annihilated by the gradient
    $\nabla f(\widetilde{p})$, the image of $\nabla|_{\widetilde{p}}$ is
    contained in $(T_{\mathbb{P}^n,p} / T_{X,p})^*$ and
    $\dim_{\mathbb{C}} (I_X)_2 - \kept_{ \{ p \}}(X) \leqslant \codim X$.
    \qedhere
  \end{compactenum}
\end{proof}

As an application, we characterize the projective subvarieties having
quadratic persistence one.

\begin{corollary}
  \label{c:minqp}
  The quadratic persistence of a complex subvariety $X \subseteq \mathbb{P}^n$
  equals one if and only if the vector space $(I_X)_2$ is nonempty and the
  hypersurfaces corresponding to a basis for $(I_X)_2$ intersect transversely
  at a generic point in $X$.
\end{corollary}

\begin{proof}
  For notational brevity, set $I \coloneqq I_X$.  By definition, the equality
  $\qp(X) = 0$ is equivalent to the vector space $I_2$ being empty. Hence, the
  equality $\qp(X) = 1$ ensures that, for a generic point $p \in X$, we have
  $\kept_{ \{ p \} }(X) = 0$. Let $\widetilde{p} \in \mathbb{A}^{n+1}$ be an
  affine representative of the point $p \in \mathbb{P}^n$. If
  $m \coloneqq \dim_{\mathbb{C}} (I_2)$ and the polynomials
  $f_1, f_2, \dotsc, f_m$ form a basis for $I_2$, then Part~\emph{vi} of
  Lemma~\ref{l:qpGen} establishes that the gradients
  $\nabla f_i(\widetilde{p})$, for all $1 \leqslant i \leqslant m$, are
  linearly independent.
\end{proof}

We assemble the number of quadrics kept under successive inner projections
into a sequence. For a non-degenerate irreducible complex subvariety
$X \subseteq \mathbb{P}^n$, set
$\kept_{\!j}(X) \coloneqq \kept_{\Gamma_{\!j}}(X)$ where $\Gamma_{\!j}$ is any
general set of closed points on $X$ having cardinality $j$ and let
$\kept(X) \coloneqq \bigl( \kept_{0}(X), \kept_{1}(X), \kept_{2}(X), \dotsc
\bigr) \in \mathbb{N}^{\mathbb{N}}$. Part~\emph{iv} of Lemma~\ref{l:qpGen}
proves that the sequence $\kept(X)$ is independent of the choice of the
general sets.  We verify that $\kept(X)$ is a strictly-convex integer
partition with distinct parts.

\begin{proposition}
  \label{p:convex}
  For a non-degenerate irreducible complex subvariety
  $X \subseteq \mathbb{P}^n$, the sequence $\kept(X)$ of nonnegative integers
  is decreasing with $\qp(X)$ nonzero entries that, for all
  $0 < j \leqslant \qp(X)$, satisfies
  \[
    2 \, \kept_{\!j}(X) < \kept_{\!j-1}(X) + \kept_{\!j+1}(X) \, .
  \]
\end{proposition}

\begin{proof}
  Again for brevity, let $k \coloneqq \qp(X)$ and let
  $\kept_{\!j} \coloneqq \kept_{\!j}(X)$ for all $j \in \mathbb{N}$.  Choose a
  general set $\Gamma_k \coloneqq \{ p_1, p_2, \dotsc, p_k \} \subseteq X$ of
  closed points and set $\Gamma_{\!j} \coloneqq \{ p_1, p_2, \dotsc, p_{\!j} \}$.
  Part~\emph{ii} of Lemma~\ref{l:qpGen} demonstrates that
  $\kept_{\!j} \geqslant \kept_{\!j+1}$ and the definition of quadratic
  persistence implies that $\kept_{\!j} = 0$ if and only if
  $j \geqslant \qp(X)$.  If $\kept_{\!j} = \kept_{\!j-1}$ for some
  $0 < j \leqslant k$, then Part~\emph{vi} of Lemma~\ref{l:qpGen} produces a
  nonzero quadratic polynomial $f \in I_X$ which is singular at the closed
  point $p_{\!j}$.  Since the singular locus of $f$ is a linear subspace of
  $\mathbb{P}^n$, we conclude that $\pi_{\Gamma_{\!j-1}}(X)$ is degenerate,
  which contradicts the hypothesis that $X$ is non-degenerate or the genercity
  of $\Gamma_k$.  It follows that $\kept_{\!j-1} > \kept_{\!j}$ for all
  $0 < j \leqslant k$.

  To prove convexity, it suffices to show that, for all $0 < j \leqslant k$,
  the difference $\Delta \kept_{\!j} \coloneqq \kept_{\!j} - \kept_{\!j-1}$ is
  strictly increasing.  Let $W_{\!j}$ denote the linear subspace of quadrics
  in $(I_X)_2$ that are singular at the point $p_{\!j}$.  Hence, Part~\emph{i}
  of Lemma~\ref{l:qpGen} gives
  \[
    \Delta \kept_{\!j} = \dim_{\mathbb{C}} (W_1 \cap W_2 \cap \dotsb \cap W_{\!j})
    - \dim_{\mathbb{C}} (W_1 \cap W_2 \cap \dotsb \cap W_{\!j-1}) \, .
  \]
  Write $V \coloneqq \bigcap_{l=1}^{j-1} W_l \subseteq (I_X)_2$ and set
  $W_i^{\perp} \coloneqq \left\{ \psi \in V^* \; \middle| \;
    \text{$\psi(f) = 0$ for all $f \in W_i \cap V$} \right\}$ where $i = j$ or
  $i = j+1$.  It follows that $\Delta \kept_{\!j+1} > \Delta \kept_{\!j}$ is
  equivalent to
  $\dim_{\mathbb{C}} (\smash{W_{\!j}^\perp}) > \dim_{\mathbb{C}} \bigl(
  \smash{W_{\!j+1}^\perp} + \smash{W_{\!j}^\perp}) / \smash{W_{\!j}^\perp} \bigr)$.
  The latter relation holds if and only if
  $\smash{W_{\!j}^\perp} \cap \smash{W_{\!j+1}^\perp} \neq 0$ which, by duality,
  is the same as saying that $(V \cap W_{\!j}) + (V \cap W_{\!j+1}) \neq V$.  We
  establish this last inequality by contradiction.  Assuming that
  $(V \cap W_{\!j}) + (V \cap W_{\!j+1}) = V$, every quadratic polynomial
  $f \in V$ can be written as $f = f_{\!j} + f_{\!j+1}$ where
  $f_i \in V \cap W_i$.  Since $f$ is homogeneous, we see that it vanishes on
  the entire line passing through the closed points $p_{\!j}$ and $p_{\!j+1}$.  In
  other words, each quadratic polynomial in $V$ vanishes on the secant variety
  of $\pi_{\Gamma_{\!j-1}}(X)$.  However, Lemma~2.2 in \cite{LM} confirms that
  this contradicts the hypothesis that $X$ is non-degenerate.
\end{proof}

\begin{remark}
  The sequence $\kept(X)$ is closely related to the gap vector introduced
  in~\cite{BIJV}. To be more explicit, we must assume that
  $X \subseteq \mathbb{P}^n$ is a non-degenerate totally-real variety having
  codimension $c$.  If
  $g(X) \coloneqq \bigl( g_1(X), g_2(X), \dotsc, g_c(X) \bigr)$ is the gap
  vector from Definition~1.1 of \cite{BIJV}, then Theorem~1.6 in \cite{BIJV}
  implies that
  \[
    g_{\!j}(X) - \kept_{\!j}(X) = \tbinom{c+1}{2} - \dim_{\mathbb{C}} (I_X)_2 -
    \tbinom{c-j+1}{2} \, ,
  \]
  for all $1 \leqslant j \leqslant c$, and Theorem~1.7 in \cite{BIJV} proves
  that $\lambda(X)$ is an integer partition with distinct parts.  However, the
  convexity of $\lambda(X)$ reveals that the gap vector is also convex. For
  instance, the existence of an index $i$ such that $g_{i-1}(X) = g_i(X)$
  implies that $g_{\!j}(X) = 0$ for all $1 \leqslant j \leqslant i$.
\end{remark}

Using the properties of the sequence $\kept(X)$, we see that the quadratic
persistence bounds the dimension of the linear subspace of quadrics in the
defining ideal of a variety.

\begin{corollary}
  \label{c:qpIneq}
  For a non-degenerate irreducible complex subvariety
  $X \subseteq \mathbb{P}^n$, we have
  \[
    \binom{\qp(X)+1}{2} \leqslant \dim_{\mathbb{C}} (I_{X})_2 \leqslant \qp(X)
    \, \codim(X) - \binom{\qp(X)}{2} \, .
  \]
\end{corollary}

\begin{proof}
  Let $c \coloneqq \codim(X)$, let $k \coloneqq \qp(X)$, and let
  $\kept_{\!j} \coloneqq \kept_{\!j}(X)$ for all $j \in \mathbb{N}$.  We first
  bound the difference $\Delta \kept_{\!j}$ for all $0 < j \leqslant
  k$. Choose a general set
  $\Gamma_k \coloneqq \{ p_1, p_2, \dotsc, p_k \} \subseteq X$ of closed
  points.  Setting $\Gamma_{\!j} \coloneqq \{ p_1, p_2, \dotsc, p_{\!j} \}$
  for all $0 < j \leqslant k$, we see that
  $\codim \pi_{\Gamma_{\!j-1}}(X) = c - (j-1)$ and Part~\emph{vii} in
  Lemma~\ref{l:qpGen} gives
  \begin{align*}
    \Delta \kept_{\!j}
    = \kept_{\!j} - \kept_{\!j-1} 
    &= \bigl( \dim_{\mathbb{C}} (I_{\pi_{\Gamma_{\!j-1}}(X)})_2 - \kept_{\!j-1} \bigr) -
      \bigl( \dim_{\mathbb{C}} I_{\pi_{\Gamma_{\!j-1}}(X)} - \kept_{\!j} \bigr) \\
    &=  - \bigl( \dim_{\mathbb{C}} (I_{\pi_{\Gamma_{\!j-1}}(X)})_2 - \kept_{\!j}
      \bigr) \geqslant - \codim \pi_{\Gamma_{\!j-1}}(X) = (j-1) - c \, . 
  \end{align*}
  Combined with Proposition~\ref{p:convex}, we deduce that
  $-1 \geqslant \Delta \kept_{\!j} \geqslant (j-1) - c$.  By definition, we
  have
  $\dim_{\mathbb{C}} (I_X)_2 = \kept_0 = (\kept_0 - \kept_1) + (\kept_1 -
  \kept_2) + \dotsb + (\kept_{k-1} - \kept_k) = \sum_{j=1}^k (- \Delta
  \kept_{\!j})$, so
  \[
    \binom{k+1}{2} = \sum_{j=1}^k j \leqslant \dim_{\mathbb{C}} (I_X)_2
    \leqslant \sum_{j=1}^k c - (j-1) = k \, c - \binom{k}{2} \, . \qedhere
  \]
\end{proof}

%% ---------------------------------------------------------------------------
\subsection*{Varieties with large quadratic persistence}

The bounds in Corollary~\ref{c:qpIneq} allow us to classify the subvarieties
with maximal quadratic persistence.  This classification simultaneously shows
that the upper and lower bounds can coincide.

\begin{theorem}
  \label{t:maxqp}
  For a non-degenerate irreducible complex subvariety
  $X \subseteq \mathbb{P}^n$, we have the equality $\qp(X) = \codim(X)$ if and
  only if $X$ has minimal degree, that is $\deg(X) = 1 + \codim(X)$.
\end{theorem}

\begin{proof}
  Set $c \coloneqq \codim(X)$ and $k \coloneqq \qp(X)$.  The hypothesis
  $k = c$ implies that $kc - \binom{k}{2} = \binom{k+1}{2}$, so the bounds in
  Corollary~\ref{c:qpIneq} are equal.  It is well-known, going back to
  G.~Castelnuovo, that the equality
  $\dim_{\mathbb{C}} (I_X)_2 = \binom{c+1}{2}$ is equivalent to $X$ being a
  variety of minimal degree; see Corollary~5.8 in \cite{Zak}.  Conversely,
  assuming that $X$ has minimal degree, we have
  $\dim_{\mathbb{C}} (I_X)_2 = \binom{c+1}{2}$ and the bounds in
  Corollary~\ref{c:qpIneq} become
  $0 \leqslant \binom{c+1}{2} - \binom{k+1}{2} \leqslant kc - k^2$ or
  $0 \leqslant (c-k)(c+k+1) \leqslant (c-k)k$.  Since Part~\emph{i} in
  Lemma~\ref{l:qpBasic} establishes that $c-k \geqslant 0$, the strict
  inequality $c - k > 0$ would imply that $c+k+1 \leqslant k$ which is absurd.
  We conclude that $c = k$ when $X$ has minimal degree.
\end{proof}

To expand on this classification, we look at another prominent numerical
invariant of a variety.  Following Section~3 of \cite{BSV16}, the
\define{quadratic deficiency} of the projective subvariety
$X \subseteq \mathbb{P}^n$ is
$\varepsilon(X) \coloneqq \binom{\codim(X) + 1}{2} - \dim_{\mathbb{C}}
(I_X)_2$.  From this perspective, Theorem~\ref{t:maxqp} proves that
$\varepsilon(X) = 0$ if and only if $\qp(X) = \codim(X)$.  For the
subvarieties having small positive quadratic deficiency, we have the following
one-way implication.

\begin{proposition}
  \label{p:one}
  For a non-degenerate irreducible complex subvariety
  $X \subseteq \mathbb{P}^n$ such that either $\varepsilon(X) = 1$ or
  $\varepsilon(X) = 2$, we have $\qp(X) = \codim(X) - 1$.
\end{proposition}

\begin{proof}
  Let $c \coloneqq \codim(X)$ and let $k \coloneqq \qp(X)$.  The inequalities
  in Corollary~\ref{c:qpIneq} are equivalent to
  $\binom{c - k + 1}{2} \leqslant \varepsilon(X) \leqslant \binom{c + 1}{2} -
  \binom{k + 1}{2}$.  From this lower bound on $\varepsilon(X)$ and our
  hypothesis on $\varepsilon(X)$, we deduce that
  $(c-k+1)(c-k) \leqslant 2\, \varepsilon(X) \leqslant 4$. Together
  Part~\emph{i} in Lemma~\ref{l:qpBasic} and Theorem~\ref{t:maxqp} establish
  that $c-k \geqslant 1$.  Since $c-k \in \mathbb{Z}$, we infer that
  $c-k+1 = 2$ and $c-k = 1$, so $k = c-1$.
\end{proof}

\begin{remark}
  \label{r:Zak}
  Proposition~5.10 in \cite{Zak} proves that a projective subvariety
  $X \subset \mathbb{P}^n$ with $\varepsilon(X) = 1$ is a hypersurface of
  degree at least $3$ or linearly-normal variety such that
  $\deg(X) = 2 + \codim(X)$.  Corollary~1.4 in \cite{Par} proves that, for a
  subvariety $X \subset \mathbb{P}^n$ satisfying $\codim(X) \geqslant 3$ and
  $\varepsilon(X) = 2$, the pair
  $\bigl( \deg(X), \operatorname{depth}(X) \bigr)$ is either
  $\bigl( 2 + \codim(X), \dim(X) \bigr)$ or
  $\bigl( 3 + \codim(X), 1 + \dim(X) \bigr)$.
\end{remark}

When $\varepsilon(X) = 1$, Proposition~\ref{p:one} shows that the upper bound
in Corollary~\ref{c:qpIneq} is achieved.  It also shows that the lower bound
is attained when $\varepsilon(X) = 2$ and $\codim(X) = 2$, which means that
$X$ is a complete intersection of two quadrics.  Extending both cases, the
subsequent family of varieties have almost maximal quadratic persistence and a
minimal number of quadratic generators.

\begin{example}[Extremal varieties with almost maximal quadratic persistence]
  \label{e:almost}
  Let $X \subset \mathbb{P}^n$ be the intersection of a general hypersurface
  of degree at least two with a variety $X' \subset \mathbb{P}^n$ of minimal
  degree.  If the hypersurface has degree greater than two, then it follows
  that $(I_X)_2 = (I_{X'})_2$ and
  \[
    \varepsilon(X) = \tbinom{\codim(X)+1}{2} - \tbinom{\codim(X)}{2} =
    \codim(X) \, .
  \]
  When the hypersurface has degree two, we have
  $\dim_{\mathbb{C}}(I_X)_2 = 1 + \dim_{\mathbb{C}}(I_{X'})_2$ and
  \[
    \varepsilon(X) = \tbinom{\codim(X')+2}{2} - \tbinom{\codim(X')+1}{2} - 1 =
    \codim(X) - 1 \, .
  \]
  In both cases, parts~\emph{i} and \emph{ii} in Lemma~\ref{l:qpBasic} show
  that $\codim(X') + 1 = \codim(X) \geqslant \qp(X) \geqslant \qp(X')$.  Using
  Theorem~\ref{t:maxqp} twice, we also see that $\codim(X) > \qp(X)$ and
  $\qp(X') = \codim(X')$.  Thus, we surmise that
  $\qp(X) = \codim(X') = \codim(X) - 1$.
\end{example}

Under the additional assumption that $X \subseteq \mathbb{P}^n$ is
arithmetically Cohen--Macaulay, the two extremal possibilities become the only
options.  To explain this, we start on the $0$\nobreakdash-dimensional case
with an analogue of the Strong Castelnuovo Lemma; see Theorem~3.c.6 in
\cite{Gre}.

\begin{lemma}
  \label{l:0dim}
  Let $n \geqslant 2$ and let $X \subset \mathbb{P}^n$ be a set of closed
  points in linearly general position. We have $\qp(X) \geqslant n - 1$ if and
  only if $X$ lies on a rational normal curve.
\end{lemma}

\begin{proof}
  When $X$ lies on a rational normal curve $C \subset \mathbb{P}^n$,
  Part~\emph{ii} of Lemma~\ref{l:qpBasic} demonstrates that
  $\qp(X) \geqslant \qp(C)$ and Theorem~\ref{t:maxqp} establishes that
  $\qp(C) = \codim(C) = n - 1$.

  For the other implication, suppose that $\qp(X) \geqslant n-1$.  We proceed
  by induction on $n$.  For $n = 2$, we have $\qp(X) \geqslant 1$ if and only
  if $X$ lies on a quadratic curve, which is a rational normal curve in
  $\mathbb{P}^2$.  We now assume that $n > 2$.  Every set of $n + 3$ closed
  points in linearly general position in $\mathbb{P}^n$ lies on a unique
  rational normal curve; see Theorem~1.18 in \cite{Har}.  Hence, we may also
  assume that $\abs{X} > n+3$.  Let
  $\Gamma \coloneqq \{ p_1, p_2, \dotsc, p_{n+3} \} \subset X$ be a set of
  closed points and let $C$ be the unique rational normal curve containing
  $\Gamma$.  For any point $p_i \in \Gamma$, the set
  $X' \coloneqq \pi_{\{p_i\}}(X \setminus \{p_i\})$ is in linearly general
  position.  Part~\emph{v} of Lemma~\ref{l:qpGen} shows that
  $\qp(X') \geqslant \qp(X) - 1 \geqslant n - 2$ and the induction hypothesis
  shows that the set $X'$ is contained in a rational normal curve
  $C' \subset \mathbb{P}^{n-1}$.  As $C'$ and $\pi_{\{p_i\}}(C)$ are rational
  normal curves passing through the $n+2$ points in
  $\pi_{\{p_i\}}(\Gamma \setminus \{p_i\})$, we see that
  $C' = \pi_{\{p\}}(C)$.  It follows that, for all $p_i \in \Gamma$, the ideal
  $I_X$ contains the ideal of the cone over $\pi_{\{p_i\}}(C)$ with vertex
  $p_i$.

  We next describe the ideals of $\pi_{\{p_i\}}(C)$ more explicitly; compare
  with Exercise~1.25 in \cite{Har}.  Fix two distinct points
  $p_1, p_2\in \Gamma$ and an isomorphism
  $\nu \colon \mathbb{P}^1 \rightarrow C$. Choose coordinates on
  $\mathbb{P}^1$ such that $\nu([0 : 1]) = p_1$ and $\nu([1 : 0]) = p_2$ and
  choose coordinates on $\mathbb{P}^n$ such that the morphism $\nu$ is given
  by $[t_0 : t_1] \mapsto [t_0^n : t_0^{n-1} t_1^{} : \dotsb : t_1^n]$.  In
  these coordinates, the ideals $I_C$, $I_{\pi_{\{p_1\}}(C)}$, and
  $I_{\pi_{\{p_2\}}(C)}$ are given by the maximal minors of the matrices
  \begin{align*}
    &
      \begin{bmatrix}
        x_0 & x_1 & \dotsb & x_{n-1} \\[-2pt]
        x_1 & x_2 & \dotsb & x_{n} \\
      \end{bmatrix} \, ,
    &&\begin{bmatrix}
      x_1 & x_2 & \dotsb & x_{n-1} \\[-2pt]
      x_2 & x_3 & \dotsb & x_{n} \\
    \end{bmatrix} \, ,
    &&\text{and}
    &&\begin{bmatrix}
      x_0 & x_1 & \dotsb & x_{n-2} \\[-2pt]
      x_1 & x_2 & \dotsb & x_{n-1} \\
    \end{bmatrix} 
  \end{align*}
  respectively.  Setting
  $J \coloneqq I_{\pi_{\{p_1\}}(C)} + I_{\pi_{\{p_2\}}(C)}$, the previous
  paragraph proves that $J \subseteq I_X$.  For all
  $1 \leqslant j \leqslant n-1$, we have
  $x_{\!j} \, x_0 \, x_n = x_{\!j-1} \, x_1 \, x_n = x_{\!j} \, x_{n-1} \, x_1
  \pmod{J}$, so $x_{\!j} ( x_0 x_n - x_1 x_{n-1}) \in J$ and
  $I_C = J : \ideal{x_1, x_2, \dotsc, x_{n-1}}$.  Hence, the reduced scheme
  defined by $J$ is the union of the rational normal curve $C$ and the line
  through the closed points $p_1$ and $p_2$.  Since no three points of $X$ are
  collinear and $X \subset \operatorname{V}(J)$, we deduce that $X \subset C$.
\end{proof}

\begin{remark}
  One cannot extend the argument in Lemma~\ref{l:0dim} to higher-dimensional
  varieties by constructing the determinantal representations of rational
  normal scrolls as in Example~9.15 of \cite{Har}.  For instance, consider the
  irreducible curve $X$ lying on the Veronese surface
  $\nu_2(\mathbb{P}^2) \subset \mathbb{P}^5$ obtained by intersecting
  $\nu_2(\mathbb{P}^2)$ with a general cubic hypersurface. For all closed
  points $p \in X$, the projection $\pi_{ \{ p \}}(X)$ is contained in a
  $2$\nobreakdash-dimensional rational normal scroll.  However, the curve $X$
  is not contained in a $2$\nobreakdash-dimensional rational normal scroll,
  because the quadrics in $I_X$ define $\nu_2(\mathbb{P}^2)$.
\end{remark}

We now turn to the higher-dimensional situation.  

\begin{theorem}
  \label{t:acm}
  For a non-degenerate irreducible complex subvariety
  $X \subseteq \mathbb{P}^n$ that is arithmetically Cohen--Macaulay, we have
  $\qp(X) = \codim(X) - 1$ if and only if either $\varepsilon(X) = 1$ or $X$
  is a codimension-one subvariety in a variety of minimal degree.
\end{theorem}

\begin{proof}
  Proposition~\ref{p:one} and Example~\ref{e:almost} prove one direction.  For
  the other direction, assume that $\qp(X) = \codim(X) - 1$.  We proceed by
  induction on $d \coloneqq \dim(X)$.  The Bertini Theorem (Theorem~17.16 in
  \cite{Har}) implies that the intersection of $X$ with $d$ general
  hyperplanes yields a set $Z$ of closed points in linearly general position.
  Since Lemma~\ref{l:0dim} demonstrates that the homogeneous ideal $I_Z$
  contains the ideal of a rational normal curve in $C$, Corollary~1.26 in
  \cite{AN} proves that the linear strand in the minimal free resolution of
  the ideal $I_Z$ has at least $\codim(C) = \codim(Z) - 1$ nonzero terms.
  Because the $d$ general hyperplanes form a regular sequence, Lemma~2.19 in
  \cite{AN} implies that the minimal free resolution of $I_X$ also has at
  least $\codim(Z) - 1 = \dim(X) - 1$ nonzero terms.  Applying Green's
  $K_{p,1}$\nobreakdash-Theorem (Theorem~3.c.1 in \cite{Gre}) shows that
  either $\deg(X) \leqslant 2 + \codim(X)$ or the variety $X$ lies in a
  variety of minimal degree having dimension equal to $1+\dim(X)$.
  Theorem~\ref{t:maxqp} precludes the possibility that
  $\deg(X) = 1 + \codim(X)$, so the classification of varieties with quadratic
  deficiency $1$ appearing in Remark~\ref{r:Zak} proves that
  $\varepsilon(X) = 1$.
\end{proof}

\begin{remark}
  Having almost maximal quadratic persistence dictates how many quadrics are
  kept under succesive projections.  When $\qp(X) = \codim(X) - 1$, the
  sequence $\kept \coloneqq \kept(X)$ is given by
  \[
    \Delta \kept_{\!j} = \kept_{\!j} - \kept_{\!j-1} =
    \begin{cases}
      \codim(X) - j + 1 & \text{if $j \leqslant \codim(X) - \varepsilon(X)$,}
      \\[-2pt]
      \codim(X) - j & \text{if $j > \codim(X) - \varepsilon(X)$.}
    \end{cases}
  \]
\end{remark}

\begin{remark}
  Our results and examples of subvarieties with large quadratic persistence
  suggest a dichotomy.  Either the ideal intrinsically has many quadratic
  polynomials or the variety has small codimension in another variety with
  large quadratic persistence.  Perhaps the strongest statement of this form,
  consistent with our work, is the following: for any $d \in \mathbb{N}$, the
  quadratic persistence of a non-degenerate irreducible complex subvariety
  $X \subseteq \mathbb{P}^n$ is at least $\codim(X) - d$ if and only if
  $\varepsilon(X) < \binom{d+2}{2}$ or $X$ is a hypersurface in a variety $Y$
  with quadratic persistence $\codim(Y) - d + 1$.  It would be interesting to
  determine whether a statement of this form is true.
\end{remark}

%% ---------------------------------------------------------------------------
\subsection*{Lower bounds on Pythagoras numbers}

To link the quadratic persistence of a subvariety to its Pythagoras number, we
focus on an irreducible totally-real subvariety $X \subseteq \mathbb{P}^n$.
When working over the real numbers, we typically focus on the real points in a
variety. The next proposition shows that, for irreducible totally-real
varieties, the quadratic persistence is insensitive to the distinction between
real and complex points.

\begin{lemma}
  \label{l:qpReal}
  For an irreducible totally-real subvariety $X \subseteq \mathbb{P}^n$, the
  quadratic persistence of $X$ is equal to the smallest cardinality of a
  finite set $\Gamma$ of real points in $X$ such that the homogeneous ideal
  $I_{\pi_\Gamma(X)}$ contains no quadratic polynomials.
\end{lemma}

\begin{proof}
  The complexification of the real variety $X$ is the complex variety
  $X_{\mathbb{C}} \coloneqq X \times_{\operatorname{Spec}(\mathbb{R})}
  \operatorname{Spec}(\mathbb{C})$ and Part~\emph{iv} in Lemma~\ref{l:qpGen}
  establishes that $\qp(X_{\mathbb{C}})$ is the cardinality of the smallest
  general set $\Gamma$ of closed points in $X_{\mathbb{C}}$ such that the
  ideal $I_{\pi_\Gamma(X_{\mathbb{C}})}$ contains no quadratic polynomials.
  Since the variety $X$ is totally real if and only if the set $X(\mathbb{R})$
  of real points is Zariski dense, we may assume that the set $\Gamma$ of
  closed points that determines the quadratic persistence contains only real
  points.
\end{proof}

The strategy for creating lower bounds on the Pythagoras number involves
restricting to faces in the cone $\Sos_X$.  The crucial observation, for which
variants already appear in Proposition~1.1 in \cite{BIK}, Theorem~1.6 in
\cite{BIJV}, and Proposition~3.3 in \cite{Sch17}, is the following lemma.  For
any subset $\Gamma$ of closed points in $X$, the projection
$\pi_\Gamma \colon \mathbb{P}^n \dashrightarrow \mathbb{P}^{n-k}$ induces a
monomorphism
$\pi_\Gamma^{\sharp} \colon \mathbb{R}[y_0,y_1, \dotsc,
y_{n-k}]/I_{\pi_\Gamma(X)} \to \mathbb{R}[x_0, x_1, \dotsc, x_n]/I_X$ between
the homogeneous coordinate rings.

\begin{lemma}
  \label{l:face}
  Let $X \subseteq \mathbb{P}^n$ be a real subvariety.  For any subset
  $\Gamma$ of real points in $X$, the monomorphism $\pi_\Gamma^{\sharp}$
  identifies the sums-of-squares cone $\Sos_{\pi_\Gamma(X)}$ with the face in
  $\Sos_X$ consisting of all quadratic polynomials vanishing at the points in
  $\Gamma$.
\end{lemma}

\begin{proof}
  Let $F \subset \Sos_X$ be the face of the sums-of-squares cone in
  $R \coloneqq \mathbb{R}[x_0, x_1, \dotsc, x_n]/I_X$ consisting of all
  quadratic polynomials vanishing at the points in $\Gamma$.  As
  $\pi_\Gamma^\sharp$ is homomorphism of $\mathbb{N}$\nobreakdash-graded
  rings, we see that $\pi_\Gamma^\sharp(\Sos_{\pi_\Gamma(X)}) \subseteq F$.
  Consider $f = g_1^2 + g_2^2 + \dotsb + g_r^2 \in F$ and fix $p \in \Gamma$.
  Since $f$ vanishes at the real point $p$, we see that, for all
  $1 \leqslant i \leqslant r$, the element $g_i$ also vanishes at $p$.  Hence,
  the elements $f, g_1, g_2, \dotsc, g_r$ all lie in the image of the map
  $\pi_\Gamma^\sharp$, so we have
  $F \subseteq \pi_\Gamma^\sharp(\Sos_{\pi_\Gamma(X)})$.
\end{proof}

Theorem~\ref{t:3} rephrases this observation in terms of the quadratic
persistence and provided our original motivation for Definition~\ref{d:qp}.

\begin{proof}[Proof of Theorem~\ref{t:3}]
  Set $k \coloneqq \qp(X)$.  Lemma~\ref{l:qpReal} ensures that exists a set
  $\Gamma \coloneqq \{p_1, p_2, \dotsc, p_k \}$ of real points in $X$ such
  that the ideal $I_{\pi_{\Gamma}(X)}$ contains no quadratic polynomials.  The
  non-degeneracy of $X$ implies the non-degeneracy of $\pi_\Gamma(X)$, so the
  cone $\Sos_{\pi_\Gamma(X)}$ is equal to the sums-of-squares cone in
  $\mathbb{P}^{n-k}$.  Since $\py(\mathbb{P}^{n-k}) = n-k+1$,
  Lemma~\ref{l:face} establishes that $\py(X) \geqslant n+1-k$.  Lastly,
  Part~\emph{i} of Lemma~3.2 proves that $k \leqslant \codim(X)$, so
  $\py(X) \geqslant 1 + n - \codim(X) = 1 + \dim(X)$.
\end{proof}

As an immediate consequence, we can strengthen Theorem~1.1 in \cite{BPSV}.

\begin{corollary}
  \label{c:mindeg}
  For any non-degenerate irreducible totally-real subvariety
  $X \subseteq \mathbb{P}^n$, we see that $X$ is a variety of minimal degree
  if and only if $\py(X) = 1 + \dim(X)$.
\end{corollary}

\begin{proof}
  Suppose that $X$ is a variety of minimal degree.  Theorem~\ref{t:maxqp}
  shows that $\qp(X) = \codim(X)$. Combining Corollary~\ref{c:gl} and
  Theorem~\ref{t:3} gives
  $1 + \dim(X) \geqslant \py(X) \geqslant 1 + \dim(X)$.  Conversely, suppose
  that $\py(X) = 1 + \dim(X)$. Combining Part~\emph{i} of
  Lemma~\ref{l:qpBasic} and Theorem~\ref{t:3} gives
  $\codim(X) \geqslant \qp(X) \geqslant \codim(X)$ and Theorem~\ref{t:maxqp}
  shows that $X$ is a variety of minimal degree.
\end{proof}

\begin{proof}[Proof of Theorem~\ref{t:4}]
  Theorem~\ref{t:maxqp} proves that conditions \emph{a} and \emph{c} are
  equivalent, and Corollary~\ref{c:mindeg} proves between conditions
  \emph{a} and \emph{b} are equivalent.
\end{proof}

We end this section by describing the arithmetically Cohen--Macaulay varieties
with almost minimal Pythagoras numbers.

\begin{corollary}
  \label{c:acmPy}
  Let $X \subseteq \mathbb{P}^n$ be a non-degenerate irreducible totally-real
  subvariety.  Assuming that $\py(X) = 2 + \dim(X)$, we have
  $\qp(X) = \codim(X) - 1$.  If $X$ is arithmetically Cohen--Macaulay and
  $\qp(X) = \codim(X) - 1$, then we also have $\py(X) = 2 + \dim(X)$.
\end{corollary}

\begin{proof}
  Assume $\py(X) = 2 + \dim(X)$. Theorem~\ref{t:3} yields
  $ 2 + \dim (X) = \py(X) \geqslant n + 1 - \qp(X)$, so we obtain the lower
  bound $\qp(X) \geqslant \codim(X) - 1$. Part~\emph{i} of
  Lemma~\ref{l:qpBasic} provides the upper bound $\codim(X) \geqslant \qp(X)$.
  Since Theorem~\ref{t:maxqp} and Corollary~\ref{c:mindeg} show that
  $\qp(X) = \codim(X)$ if and only if $\py(X) = 1 + \dim(X)$, we conclude that
  $\qp(X) = \codim(X) - 1$.

  Suppose that $X$ is arithmetically Cohen--Macaulay and
  $\qp(X) = \codim(X) - 1$.  Theorem~\ref{t:3} yields the lower bound
  $\py(X) \geqslant 2 + \dim(X)$.  To give the matching upper bound,
  Theorem~\ref{t:acm}, together with Remark~\ref{r:Zak}, divides the analysis
  into two cases: if $X$ is a subvariety having codimension $1$ in a variety
  of minimal degree, then Theorem~\ref{t:2reg} proves that
  $\py(X) \leqslant 2 + \dim(X)$, and if $\deg(X) = 2 + \codim(X)$, then
  Corollary~\ref{c:acmUp} proves that $\py(X) \leqslant 2 + \dim(X)$.
\end{proof}

\begin{proof}[Proof of Theorem~\ref{t:5}]
  Theorem~\ref{t:acm} proves that conditions \emph{a} and \emph{c} are
  equivalent, and Corollary~\ref{c:acmPy} proves between conditions \emph{a}
  and \emph{b} are equivalent.
\end{proof}

%%%%%%%%%%%%%%%%%%%%%%%%%%%%%%%%%%%%%%%%%%%%%%%%%%%%%%%%%%%%%%%%%%%%%%%%%%%%%%
\section{Quadratic persistence and minimal free resolutions}
\label{s:res}

\noindent
This section connects the quadratic persistence of a complex subvariety
$X \subseteq \mathbb{P}^n$ with a homological invariant of its homogeneous
coordinate ring $R \coloneqq \mathbb{C}[x_0, x_1, \dotsc, x_n]/I_X$ viewed as
a module over the polynomial ring
$S \coloneqq \mathbb{C}[x_0, x_1, \dotsc, x_n]$.  To be more precise, set
\begin{align}
  \label{d:len}
  \len(X) \coloneqq \max \left\{ j \in \mathbb{N} \; \middle| \;
  \smash{\Tor_j^{\,S}(R, \mathbb{C})_{1+j}}  \neq 0 \right\} \, .
\end{align}
In other words, the Betti table for the $S$\nobreakdash-module $R$ has
$\len(X)$ nonzero entries in its first row or the linear strand in the minimal
free resolution of the ideal $I_X$ has $\len(X)$ nonzero terms.  In contrast,
Section~8D in \cite{Eis} emphasizes the invariant $b(X) \coloneqq \len(X) + 1$
when $X$ is a curve of high degree.

\begin{remark}
  \label{r:Betti}
  The numerical invariants of a minimal free resolution can be compactly
  displayed in an array.  Following Section~1B in \cite{Eis}, the
  \define{Betti table} of an $S$\nobreakdash-module $M$ is the array whose
  $(i,j)$\nobreakdash-entry is the number
  $\dim_\mathbb{C} \Tor_j^{\,S}(M, \mathbb{C})_{i+j}$.  For a complex
  subvariety $X \subseteq \mathbb{P}^n$, the first three rows in the Betti
  table of the $S$\nobreakdash-module $R = S/I_X$ have the form
  \[
    \begin{array}{c|ccccccccccc}
      _{i} \backslash^{j} & 0 & 1 & 2 & \dotsb & \green(X) & \green(X)+1
      & \dotsb & \len(X) & \len(X)+1 & \dotsb \\ \hline
      0 & 1 & 0 & 0 & \dotsb & 0 & 0 & \dotsb & 0 & 0 & \dotsb \\[-2pt]
      1 & 0 & * & * & \dotsb & * & * & \dotsb & * & 0 & \dotsb \\[-2pt]      
      2 & 0 & 0 & 0 & \dotsb & 0 & * & \dotsb & * & * & \dotsb \\[-5pt]
      \vdots & \vdots & \vdots & \vdots & & \vdots & \vdots & &
      \vdots & \vdots &
    \end{array}
  \]
  where ``*'' denotes a positive integer.  If $\green(X)$ and $\len(X)$ are
  finite, then we have $0 \leqslant \green(X) \leqslant \len(X) \leqslant n$.
\end{remark}

Theorem~\ref{t:2} asserts that, for any non-degenerate irreducible subvariety
$X \subseteq \mathbb{P}^n$, the quadratic persistence $\qp(X)$ is bounded
below by the homological invariant $\len(X)$.  The basic plan for proving
Theorem~\ref{t:2} involves relating the linear syzygies of the variety $X$
with those of a general inner projection.  Roughly speaking, we do this by
first evaluating the matrices of linear forms, which represent the linear part
of the minimal free resolution of the homogeneous ideal $I_X$, at general
closed points of $X$.  By analyzing the vectors lying in the kernel of a
product of these complex matrices, we obtain quadratic polynomials lying in
the homogeneous ideal of the general inner projection.  The fact that these
complex matrices anti-commute is vital to the analysis.  To convert this
outline into a rigorous argument requires a fair-sized piece of homological
machinery.

For convenience, we use the Bernstein--Gelfand--Gelfand correspondence to
describe the linear part of a minimal free resolution.  Following Section~7B
in \cite{Eis}, the exterior algebra $E \cong \bigwedge (S_1)^\ast$ is the
Koszul dual of the polynomial ring $S$.  If $e_0, e_1, \dotsc, e_n$ are the
generators of $E$ dual to the variables $x_0, x_1, \dotsc, x_n$ in $S$, then
we have $e_{\!j}^2 = 0$ for all $1 \leqslant j \leqslant n$ and
$e_{\!j} \, e_{k} = - e_{k} \, e_{\!j}$ for all $1 \leqslant j < k \leqslant n$.
We equip $E$ with the $\mathbb{Z}$\nobreakdash-grading induced by setting
$\deg e_{\!j} = -1$ for all $1 \leqslant j \leqslant n$.  Although we work with
left $E$\nobreakdash-modules, any $\mathbb{Z}$\nobreakdash-graded left
$E$\nobreakdash-module $U$ can also be viewed as a
$\mathbb{Z}$\nobreakdash-graded right $E$\nobreakdash-module.  Specifically,
if $e \in E_{-j}$ and $u \in U_k$, then we have
$e \, u = (-1)^{jk} \, u \, e$.  For a finitely-generated left
$E$\nobreakdash-module $U = \bigoplus_{i \in \mathbb{Z}} U_i$, the
$\mathbb{C}$\nobreakdash-vector space dual
$U^{\ast} \coloneqq \bigoplus_{i \in \mathbb{Z}} (U_i)^{\ast}$, where
$(U_i)^{\ast} \coloneqq \Hom_\mathbb{C}(U_i, \mathbb{C})$, is naturally a
right $E$\nobreakdash-module: for all $\phi \in (U_i)^{\ast}$, all
$e \in E_{-j}$, and all $u \in U_{i+j}$, we have
$(\phi \, e)(u) = \phi (e \, u)$.  However, as a
$\mathbb{Z}$\nobreakdash-graded left $E$\nobreakdash-module where the summand
$(U^{\ast})_{-i} = (U_i)^{\ast}$ has degree $-i$, we have
$(e \, \phi)(u) = (-1)^{i \, j} \, (\phi \, e)(u) = (-1)^{i \, j} \, \phi(e \,
u)$.

The Bernstein--Gelfand--Gelfand correspondence supplies an equivalence of
categories between linear complexes of free $S$\nobreakdash-modules and
$\mathbb{Z}$\nobreakdash-graded $E$\nobreakdash-modules.  Given a
$\mathbb{Z}$\nobreakdash-graded $E$\nobreakdash-module $U$, we make the tensor
product $S \otimes_\mathbb{C} U$ into the complex of
$\mathbb{Z}$\nobreakdash-graded free $S$\nobreakdash-modules
\[
  \mathbf{L}(U) \coloneqq \quad \dotsb \longleftarrow S \otimes_\mathbb{C}
  U_{i-1} \xleftarrow{\;\; \partial_i \;\;} S \otimes_{\mathbb{C}} U_i
  \longleftarrow \dotsb \quad ,
\]
where
$\partial_i( 1 \otimes u) \coloneqq \sum_{j=0}^n x_{\!j} \otimes e_{\!j} \, u$ and
the term $S \otimes_{\mathbb{C}} U_i \cong S(-i)^{\dim_\mathbb{C} U_i}$ sits
in homological degree $i$ and is generated in degree $i$; see Section~7B in
\cite{Eis}.  By choosing bases $\{ u^{(i)}_r \}$ and $\{ u^{(i-1)}_s \}$ for
the $\mathbb{C}$\nobreakdash-vector spaces $U_{i}$ and $U_{i-1}$ so that
$e_{\!j}^{} \, u_r^{(i)} = \sum_s^{} c_{j,r,s}^{} \, u^{(i-1)}_s$ for all
$0 \leqslant j \leqslant n$ and some $c_{j,r,s} \in \mathbb{C}$, the map
$\partial_i$ is represented by a matrix of linear forms whose
$(r,s)$\nobreakdash-entry is $\sum_{j=0}^n c_{j,r,s} \, x_{\!j}$.
Proposition~7.5 in \cite{Eis} proves that $\mathbf{L}$ defines a covariant
functor and induces an equivalence from the category of
$\mathbb{Z}$\nobreakdash-graded $E$\nobreakdash-modules to the category of
linear complexes of free $S$\nobreakdash-modules.  Given a
$\mathbb{Z}$\nobreakdash-graded $E$\nobreakdash-module $U$, we identify an
element $v \in E_{-1} = (S_1)^{\ast}$ with the linear map
$v \colon S_1 \otimes_{\mathbb{C}} U \to U$ defined by
$v(x \otimes u) = v(x) \, u$.  Furthermore, for all $i \in \mathbb{Z}$, scalar
multiplication $E_{-1} \otimes_{\mathbb{C}} U_{i} \to U_{i-1}$ is defined by
$v \otimes u \mapsto v \bigl( \partial_i(u) \bigr) = \sum_{j=0}^n v(x_{\!j}) \,
e_{\!j} \, u$.

Building on this equivalence, Corollary~7.11 in \cite{Eis} identifies the left
$E$\nobreakdash-module corresponding to the linear part in the minimal free
resolution of an $S$\nobreakdash-module.  Focusing on a non-degenerate
irreducible complex subvariety $X \subseteq \mathbb{P}^n$ defined by the
saturated homogeneous $S$\nobreakdash-ideal $I_X$, the strand in the minimal
free resolution of the $S$\nobreakdash-module $R \coloneqq S/I_X$
corresponding to the first row of the Betti table is $\mathbf{L}(U_X^{\ast})$
where $U_X$ is the $E$\nobreakdash-module with free presentation
\[
  0 \longleftarrow U_X \longleftarrow E(1) \otimes_\mathbb{C} \bigl( (I_X)_{2}
  \bigr)^{\ast} \xleftarrow{\;\; \alpha \;\;} E(2) \otimes_\mathbb{C}  \bigl(
  (I_X)_{3}^{} \bigr)^{\ast}
\]
and the map $\alpha$ is defined on the generators
$1 \otimes \bigl((I_X)_{3} \bigr)^{\ast} = \bigl( (I_X)_{3} \bigr)^{\ast}$ as
the dual of the multiplication map
$S_1 \otimes_\mathbb{C} (I_X)_2 \to (I_X)_3$.  It follows that there is a
canonical isomorphism $(U_X^{\ast})_{1}^{} \cong (I_X^{})_2^{}$ of
$\mathbb{C}$\nobreakdash-vector spaces and
$\dim_\mathbb{C} (U_X)_{-j} = \dim_\mathbb{C} \Tor^{\,S}_j(R,
\mathbb{C})_{1+j}^{}$ for all $j \in \mathbb{Z}$.

To help internalize this construction, we illustrate it for an accessible
projective subvariety.

\begin{example}
  \label{e:U}
  For the rational normal curve
  $C \coloneqq \nu_3(\mathbb{P}^1) \subset \mathbb{P}^3$, the saturated
  homogeneous ideal $I_C$ is minimally generated by
  $f_0 \coloneqq x_{2}^{2} - x_{1}^{} x_{3}^{}$,
  $f_1 \coloneqq x_{1}^{} x_{2}^{} - x_{0}^{} x_{3}^{}$, and
  $f_2 \coloneqq x_{1}^{2} - x_{0}^{} x_{2}^{}$ in
  $S \coloneqq \mathbb{C}[x_0, x_1, x_2, x_3]$.  Because the syzygies among
  these three quadratic binomials are freely generated by the two relations
  $x_{0} f_0 - x_{1} f_1 + x_{2} f_2 = 0$ and
  $x_{1} f_0 - x_{2} f_1 + x_{3} f_2 = 0$, the Betti table of the
  $S$\nobreakdash-module $S/I_C$ is
  \[
    \begin{array}[b]{c|ccccccccccc}
      _{i} \backslash^{j} & 0 & 1 & 2 \\ \hline
      0 & 1 & 0 & 0 \\[-2pt]
      1 & 0 & 3 & 2 \\
    \end{array} \; . 
  \]
  Choosing the ten cubic binomials
  \begin{align*}
    g_0 &\coloneqq x_{2}^{2} x_{3}^{} - x_{1}^{} x_{3}^{2} 
    & g_1 &\coloneqq x_{1}^{} x_{2}^{} x_{3}^{} - x_{0}^{} x_{3}^{2} 
    & g_2 &\coloneqq x_{1}^{2} x_{3}^{} - x_{0}^{} x_{2}^{} x_{3}^{}
    & g_3 &\coloneqq x_{2}^{3} - x_{0}^{} x_{3}^{2} \\
    g_4 &\coloneqq x_{1}^{} x_{2}^{2} - x_{0}^{} x_{2}^{} x_{3}^{} 
    & g_5 &\coloneqq x_{0}^{} x_{2}^{2} - x_{0}^{} x_{1}^{} x_{3}^{} 
    & g_6 &\coloneqq x_{1}^{2} x_{2}^{} - x_{0}^{} x_{1}^{} x_{3}^{} 
    & g_7 &\coloneqq x_{0}^{} x_{1}^{} x_{2}^{} - x_{0}^{2} x_{3}^{} \\
    g_8 &\coloneqq x_{1}^{3} - x_{0}^{2} x_{3}^{} 
    & g_9 &\coloneqq x_{0}^{} x_{1}^{2} - x_{0}^{2} x_{2}^{} \, ,
  \end{align*}
  as a basis for $(I_C)_3$, it follows that the left $E$\nobreakdash-module
  homomorphism
  $\alpha \colon \bigoplus_{i=0}^{9} E(2) \to \bigoplus_{i=0}^{2} E(1)$
  corresponds to the matrix
  \[
    \begin{bmatrix*}
      e_{3} & -e_{2} & -e_{1} & e_{2} & e_{1} & \phantom{-}e_{0} & 0 &
      \phantom{-}0_{\phantom{.}} & 0 & 0 \\[-1pt]
      0 & \phantom{-}e_{3} & \phantom{-}0_{\phantom{.}} & 0 & e_{2} &
      \phantom{-}0_{\phantom{.}} & e_{1} & \phantom{-}e_{0} & 0  & 0 \\[-1pt]
      0 & \phantom{-}0_{\phantom{.}} & \phantom{-}e_{3} & 0 &  0 & -e_{2} &
      e_{2} & -e_{1} & e_{1} & e_{0} \\
    \end{bmatrix*} \, .
  \]
  The entries in the first row of this matrix come from the four equations
  $x_{3} f_0 = g_0$, $x_{2} f_0 = - g_1 + g_3$, $x_{1} f_0 = - g_2 + g_4$, and
  $x_{0} f_0 = g_5$.  The first row of the Betti table corresponds to the left
  $E$\nobreakdash-module $U_C \coloneqq \operatorname{Coker}(\alpha)$.  From
  the given free presentation, we may verify directly that
  $\dim_{\mathbb{C}} (U_C)_{-1} = 3$, $\dim_{\mathbb{C}} (U_C)_{-2} = 2$, and
  $\dim_{\mathbb{C}} (U_C)_{-j} = 0$ for all other $j$.  In particular, the
  three standard basis vectors for the free $E$\nobreakdash-module
  $\bigoplus_{i=0}^{2} E(1)$ surject onto a $\mathbb{C}$\nobreakdash-vector
  space basis for $(U_C)_{-1}$ and the two vectors
  $\begin{bsmallmatrix*} e_0 & 0 & 0 \end{bsmallmatrix*}^{\transpose},
  \begin{bsmallmatrix*} e_1 & 0 & 0 \end{bsmallmatrix*}^{\transpose} \in
  \bigoplus_{i=0}^{2} E(1)$ surject onto a $\mathbb{C}$\nobreakdash-vector
  space basis for $(U_C)_{-2}$.
\end{example}

The choice of a closed point $p \in X$ spawns two related linear free
complexes.  The first operation extracts the linear part of the minimal free
resolution of the homogeneous coordinate ring regarded as a module over a
smaller polynomial ring. To understand this, choose an affine representative
$\widetilde{p} \in \mathbb{A}^{n+1}$ for $p \in \mathbb{P}^n$ and let $W'$ be
the kernel of the $\mathbb{C}$\nobreakdash-linear map $S_1 \to \mathbb{C}$
defined by the evaluation $f \mapsto f(\widetilde{p})$.  Setting
$S' \coloneqq \operatorname{Sym}(W')$, the rational map
$\pi_{\{ p\}} \colon \mathbb{P}^n \dashrightarrow \mathbb{P}^{n-1} \coloneqq
\operatorname{Proj}(S')$ corresponds to the inclusion $W' \hookrightarrow S_1$
of linear subspaces.  Since $S_1 = (E_{-1})^{\ast}$, the annihilator of $W'$
is generated by
$v \coloneqq \tilde{p}_0 \, e_0 + \tilde{p}_1 \, e_1 + \dotsb + \tilde{p}_n \,
e_n$ and the exterior algebra
$E' \coloneqq E / \ideal{v} \cong \bigwedge \Hom_\mathbb{C}(W', \mathbb{C})$
is Koszul dual to the polynomial ring $S'$.  By Corollary~7.12 in \cite{Eis},
the linear part of the minimal free resolution of the $S'$\nobreakdash-module
$I_{X}(1)$ is $\mathbf{L}\bigl( (U^{\ast}_{X})' \bigr)$ where
$(U^{\ast}_{X})'$ is the $E'$\nobreakdash-module
$\left\{ u \in U^{\ast}_X \; \middle| \; v \, u = 0 \right\}$. We see that
$\dim_\mathbb{C} (U^{\ast}_X)'_{j} = \dim_\mathbb{C} \Tor^{\,S'}_j(R,
\mathbb{C})_{1+j}^{}$ for all $j \in \mathbb{Z}$.

The second operation produces the subcomplex of $\mathbf{L}(U_X^{\ast})$
generated by all of the quadratic polynomials in $I_X$ that are singular at
the closed point $p \in X$.  For the affine representative
$\widetilde{p} \in \mathbb{A}^{n+1}$ of $p \in \mathbb{P}^n$, a polynomial
$f \in (I_X)_2$ is singular at $p \in X$ if and only if the evaluation of its
gradient at this affine representative vanishes:
$\nabla f(\widetilde{p}) = \mathbf{0}$.  If $J$ denotes the
$S$\nobreakdash-ideal generated by the kernel of the linear map
$\nabla|_{\widetilde{p}} \colon (I_X)_2^{} \to T_{\mathbb{P}^2,p}^*$, then
this subcomplex is $\mathbf{L}\bigl((U^{\text{sg}}_X)^{\ast}\bigr)$ where
$U^{\text{sg}}_X$ is the $E$\nobreakdash-module with free presentation
\[
  0 \longleftarrow U^{\text{sg}}_X \longleftarrow E(1) \otimes_\mathbb{C} (
  J_{2})^{\ast} \xleftarrow{\; \alpha^{\text{sg}} \;} E(2) \otimes_\mathbb{C}
  (J_{3})^{\ast}
\]
and $\alpha^{\text{sg}}$ is defined on the generators $(J_{3})^{\ast}$ as the
dual of the multiplication map $S_1 \otimes_\mathbb{C} J_2 \to J_3$.  There is
a canonical isomorphism
$\bigl( \! (U^{\text{sg}}_X)^{\ast} \! \smash{\bigr)}_{\!1}^{} \cong J_2^{}$
of $\mathbb{C}$\nobreakdash-vector spaces and $(U^{\text{sg}}_X)^{\ast}$ is a
$\mathbb{Z}$\nobreakdash-graded submodule of $U_X^{\ast}$.

We demonstrate these two operations with the twisted cubic curve.

\begin{example}
  As in Example~\ref{e:U}, let $C$ denote the rational normal curve in
  $\mathbb{P}^3$.  From the given generators of its homogeneous ideal $I_C$,
  we see that the closed point $p \coloneqq [1 : 0 : 0 : 0] \in \mathbb{P}^3$
  lies on the curve $C$.  With this choice, the homogeneous coordinate ring
  for the codomain of the linear projection away from $p$ is just
  $S' \coloneqq \mathbb{C}[x_1, x_2, x_3]$.  When viewed by restriction of
  scalars as an $S'$\nobreakdash-module, the linear part of the minimal free
  resolution of $I_X(1)$ still has the three generators $f_0, f_1, f_2$, but
  only one syzygy $x_1 f_0 - x_2 f_1 + x_3 f_2 = 0$.  On the other hand, left
  multiplication by $v \coloneqq e_0 \in E_{-1}$ on $U_C$ is equivalent,
  modulo the defining relations for $U_C$, to acting on the free
  $E$\nobreakdash-module $\bigoplus_{i=0}^2 E(1)$ via the matrix
  \[
    \begin{bmatrix}
      e_0 & 0 & 0 \\[-2pt]
      0 & 0 & 0 \\[-2pt]
      0 & 0 & 0 \\
    \end{bmatrix} \, . 
  \]
  Hence, the cokernel of left multiplication by $v$ on $U_C$ has
  $\mathbb{C}$\nobreakdash-vector space basis corresponding to the three
  vectors $\begin{bsmallmatrix*} 1 & 0 & 0 \end{bsmallmatrix*}^{\transpose},
  \begin{bsmallmatrix*} 0 & 1 & 0 \end{bsmallmatrix*}^{\transpose},
  \begin{bsmallmatrix*} 0 & 0 & 1 \end{bsmallmatrix*}^{\transpose},
  \begin{bsmallmatrix*} e_1 & 0 & 0 \end{bsmallmatrix*}^{\transpose} \in
  \bigoplus_{i=0}^2 E(1)$, so we deduce that
  $\dim_{\mathbb{C}} (U_C^{\ast})'_{1} = 3$,
  $\dim_{\mathbb{C}} (U_C^{\ast})'_{2} = 1$, and
  $\dim_{\mathbb{C}} (U_C^{\ast})'_{j} = 0$ for all other $j$.

  The only quadratic polynomial in $I_C$ that is singular at the closed point
  $p \coloneqq [1 : 0 : 0 : 0] \in \mathbb{P}^3$ is the generator
  $f_0 = x_{2}^{2} - x_{1}^{} x_3^{}$.  It follows that $J = \ideal{f_0}$ and
  Betti table of the $S$\nobreakdash-module $S/J$ is
    \[
    \begin{array}[b]{c|ccccccccccc}
      _{i} \backslash^{j} & 0 & 1  \\ \hline
      0 & 1 & 0 \\[-2pt]
      1 & 0 & 1 \\
    \end{array} \; . 
  \]
  Choosing the four cubic binomials $h_0 \coloneqq x_{3} f_0$,
  $h_1 \coloneqq x_{2} f_0$, $h_2 \coloneqq x_{1} f_0$, and
  $h_3 \coloneqq x_{0} f_0$ as a basis for $J_3$, it follows that the left
  $E$\nobreakdash-module homomorphism
  $\alpha^{\text{sg}} \colon \bigoplus_{i=0}^{3} E(2) \to E(1)$ corresponds to
  the matrix $\begin{bmatrix*}
    e_{3} & e_{2} & e_{1} & e_{0} \\
  \end{bmatrix*}$.  Thus, the first row of the Betti table corresponds to the
  left $E$\nobreakdash-module
  $U_C^{\text{sg}} \coloneqq \operatorname{Coker}(\alpha^{\text{sg}})$.  From
  the given free presentation, we may verify directly that
  $\dim_{\mathbb{C}} (U_C^{\text{sg}})_{-1} = 1$ and
  $\dim_{\mathbb{C}} (U_C^{\text{sg}})_{-j} = 0$ for all other $j$.  In
  particular, the standard basis vector in the free $E$\nobreakdash-module
  $E(1)$ surjects onto a $\mathbb{C}$\nobreakdash-vector space basis for
  $(U_C^{\text{sg}})_{-1}$.  Foreshadowing the next lemma, we also observe
  that the coimage of multiplication by $v$ on $U_C$ is spanned by the vector
  $\begin{bsmallmatrix*} 1 & 0 & 0 \end{bsmallmatrix*}^{\transpose} \in
  \bigoplus_{i=0}^2 E(1)$ and corresponds to $U_C^{\text{sg}}$.
\end{example}

Having gathered this background and notation, we record a couple of
observations.  This lemma formalizes our heuristic that evaluating matrices of
linear forms at a point on the variety relates the linear syzygies of a
variety to those of its inner projection.

\begin{lemma}
  \label{l:exterior}
  Let $p$ be a closed point in $\mathbb{P}^n$, let
  $\widetilde{p} \coloneqq (\tilde{p}_0, \tilde{p}_1, \dotsc, \tilde{p}_n) \in
  \mathbb{A}^{n+1}$ be an affine representative of $p$, and let
  $v \coloneqq \tilde{p}_0 \, e_0 + \tilde{p}_1 \, e_1 + \dotsb + \tilde{p}_n
  \, e_n$ be the corresponding element in $E$.
  \begin{compactenum}[\upshape i.]
  \item For a general closed point $p \in X$, the condition
    $(U_X^{\ast})_{i+1}^{} \neq 0$ for some $i \geqslant 1$ implies that we have
    $0 \neq v \, (U_X^{\ast})_{i+1}^{} \subset (U_X^{\ast})_{i}^{}$.
  \item For any closed point $p \in X$, the product $v \, U_X^{\ast}$ lies in
    the $E$\nobreakdash-module $(U^{\text{\upshape sg}}_X)^{\ast}$.
  \end{compactenum}
\end{lemma}

\begin{proof}
  From the definition, we see that $(U_X^{\ast})_{i}^{} = 0$ for all
  $i \leqslant 0$.
  \begin{compactenum}[\itshape i.]
  \item By choosing bases $\{ u^{(i+1)}_{r} \}$ and $\{ u^{(i)}_{s} \}$ for
    the $\mathbb{C}$\nobreakdash-vector spaces $(U_X^{\ast})_{i+1}^{}$ and
    $(U_X\!^{\ast})_{i}^{}$ satisfying
    $e_{\!j} \, u^{(i+1)}_r = \sum_s^{} c_{j,r,s} \, u_s^{(i)}$ for all
    $0 \leqslant j \leqslant n$ and some $c_{j,r,s} \in \mathbb{C}$, the
    $E$\nobreakdash-module homomorphism from $(U_X^{\ast})_{i+1}^{}$ to
    $(U_X^{\ast})_{i}^{}$ defined by multiplication with $v$ is represented
    by the matrix whose $(r,s)$\nobreakdash-entry is the number
    $\sum_{j=0}^n c_{j,r,s} \, \tilde{p}_{\!j}$.  Since $X$ is non-degenerate
    and $p \in X$ is general, this matrix is nonzero, so the image
    $v \, (U_X^{\ast})_{i+1}^{} \subseteq (U_X^{\ast})_{i}^{}$ is nonzero.
  \item By definition, the $E$\nobreakdash-module $U^{\text{sg}}_X$ is
    generated by $(U^{\text{sg}}_X)_{-1}^{} \cong J_2^{\ast}$, so the
    $E$\nobreakdash-module $(U^{\text{sg}}_X)^{\ast}$ is cogenerated by
    $\bigl( \! (U^{\text{sg}}_X)^{\ast} \! \smash{\bigr)}_{\!1}^{} \cong
    J_2^{}$.  Hence, it suffices to show that, for all $i \geqslant 2$, all
    $v' \in E_{2-i}$, and all $u \in (U_X^{\ast})_i^{}$, we have
    $v' \, v \, u \in \bigl( \! (U^{\text{sg}}_X)^{\ast} \!
    \smash{\bigr)}_{\!1}^{} \cong J_2^{}$.  This reduces to proving that
    $v \, u \in \bigl( \! (U^{\text{sg}}_X)^{\ast} \! \smash{\bigr)}_{\!1}^{}
    \cong J_2^{}$ for all $u \in (U_X^{\ast})_2^{}$.  By choosing bases
    $\{ u^{(2)}_r \}$ and $\{ u^{(1)}_s \}$ for the
    $\mathbb{C}$\nobreakdash-vector spaces $(U_X^{\ast})_2^{}$ and
    $(U_X^{\ast})_{1}^{}$ satisfying
    $e_{\!j}^{} \, u^{(2)}_r = \sum_{s}^{} c_{j,r,s}^{} \, u^{(1)}_s$ for all
    $0 \leqslant j \leqslant n$ and some $c_{j,r,s} \in \mathbb{C}$, it
    follows that
    $v \, u^{(2)}_r = \sum_{s}^{} \sum_{j=0}^n c_{j,r,s}^{} \,
    \tilde{p}_{\!j}^{} \, u^{(1)}_s$.  If the set $\{ f_s \}$ of quadratic
    polynomials in $S$ is the basis of $(I_X)_2$ corresponding the
    $\{ u^{(1)}_s \}$, then we have
    $\nabla|_{\widetilde{p}} ( v \, u^{(2)}_r ) = \sum_{s} \sum_{j=0}^n
    c_{j,r,s}^{} \, \tilde{p}_{\!j}^{} \, \nabla f_s^{} (\widetilde{p})$.
    However, the map
    $\partial_2 \colon S \otimes_\mathbb{C} U_2 \to S \otimes U_1$ generates
    the linear syzygies among the polynomials $\{ f_s \}$, so we also have
    $\sum_{s}^{} \sum_{j=0}^n c_{j,r,s}^{} \, x_{\!j}^{} \, f_s^{} = 0$.  Since
    $p \in X$ and $f_s \in I_X$, we see that $f_s(\widetilde{p}) = 0$.  Thus,
    the product rule implies that
    $\mathbf{0} = \nabla|_{\widetilde{p}} \bigl( \sum_{s} \sum_{j=0}^n
    c_{j,r,s}^{} \, x_{\!j}^{} \, \nabla f_s^{} \bigr) = \sum_{s} \sum_{j=0}^n
    c_{j,r,s}^{} \, \tilde{p}_{\!j}^{} \, \nabla f_s^{} (\widetilde{p})$, from
    which we deduce that
    $v \, u^{(2)}_{r} \in \bigl( \! (U^{\text{sg}}_X)^{\ast} \!
    \smash{\bigr)}_{\!1}^{} \cong J_2^{}$ as required. \qedhere
  \end{compactenum}
\end{proof}

With these preparations, we present a counterpart to Corollary~7.13 in
\cite{Eis} showing that length of the linear part of a minimal free resolution
can drop by at most one under a general inner projection.  Identifying the
left $E$\nobreakdash-module corresponding to the linear part in the minimal
free resolution of the image is the critical insight.  Proposition~3.16 in
\cite{AN} presents another approach using Koszul cohomology.

\begin{proposition}
  \label{p:linPart}
  Let $X \subseteq \mathbb{P}^n$ be a non-degenerate complex subvariety. For
  any subset $\Gamma$ of $k$ general closed points in $X$, we have
  $\len(X) \leqslant k + \len \bigl( \pi_{\Gamma}(X) \bigr)$.
\end{proposition}

\begin{proof}
  By construction, we have
  $\len(X) = \max \left\{ j \in \mathbb{N} \; \middle| \;
    \smash{(U^{\ast}_{X})_{j}^{}} \neq 0 \right\}$. It suffices to consider
  the case $k = 1$.  Let $p \in X$ be a general closed point, let
  $\widetilde{p} = (\tilde{p}_0, \tilde{p}_1, \dotsc, \tilde{p}_n) \in
  \mathbb{A}^{n+1}$ be an affine representative, and let
  $v \coloneqq \tilde{p}_0 \, e_0 + \tilde{p}_1 \, e_1 + \dotsb + \tilde{p}_n
  \, e_n \in E$.  Set $W'$ to be the kernel of the
  $\mathbb{C}$\nobreakdash-linear map $S_1 \to \mathbb{C}$ defined by the
  evaluation at $\widetilde{p}$ and $S' \coloneqq \operatorname{Sym}(W')$. As
  in the proof of part~\emph{vi} of Lemma~\ref{l:qpGen}, the quadratic
  polynomials in $I_X$ that lie in $W'$ are precisely the quadrics that are
  singular at the closed point $p \in X$.  It follows that
  $(I_{\pi_{\{p\}}(X)})_2^{} = J_2$ and all of their higher syzygies lie in
  $S'$.  By design, $S'$ is annihilated by $v$, so we see that
  $v \, (U^{\text{\upshape sg}}_X)^{\ast} = 0$ and
  $\len\bigl( \pi_{\{p\}}(X) \bigr) = \max\left\{ j \in \mathbb{N} \; \middle|
    \; \smash{\bigl( (U^{\text{sg}}_X)^{\ast} \bigr)_{j}^{}} \neq 0
  \right\}$. Since $\deg(v) = -1$, Lemma~\ref{l:exterior} certifies that
  $\len(X) \leqslant 1 + \len \bigl( \pi_{\{p\}}(X) \bigr)$.
\end{proof}

\begin{proof}[Proof of Theorem~\ref{t:2}]
  Let $k \coloneqq \qp(X)$.  We first claim that $k = 1$ implies that
  $\len(X) = 1$.  To see this, suppose that the polynomials
  $f_1, f_2, \dotsc, f_m$ form a basis for the $\mathbb{C}$\nobreakdash-vector
  space $(I_X)_2$.  For a general closed point $p \in X$ with affine
  representative $\widetilde{p} \in \mathbb{A}^{n+1}$, Corollary~\ref{c:minqp}
  shows that the gradients $\nabla f_{j}(\widetilde{p})$, for all
  $1 \leqslant j \leqslant m$, are linearly independent.  If the polynomials
  $f_1, f_2, \dotsc, f_m$ have a linear syzygy, then there are linear forms
  $g_1, g_2, \dotsc, g_m$ such that $\sum_{j=1}^m g_{j} f_{j} = 0$.  Taking
  the gradient and evaluating at $\widetilde{p}$ gives
  $\sum_{j=1}^m g_{j}(\widetilde{p}) \nabla f_{j}(\widetilde{p}) = 0$, so
  $g_{j}(\widetilde{p}) = 0$ for all $1 \leqslant j \leqslant m$.  Since $X$
  is non-degenerate, we deduce that all of the linear forms $g_{j}$ are
  identically zero.  Thus, there are no linear syzygies and $\len(X) = 1$.

  Now, assume that $k > 1$.  Choose a general set
  $\{ p_1, p_2, \dotsc, p_k \}$ of closed points in $X$ and, for all
  $1 \leqslant j \leqslant k$, set
  $\pi_{j} \coloneqq \pi_{\{ p_1, p_2, \dotsc, p_{j} \} }$.  Combining
  Definition~\ref{d:qp} and Lemma~\ref{l:qpGen} affirms that
  $\bigl( I_{\pi_{k}(X)} \bigr)_2 = 0$ and
  $\bigl( I_{\pi_{k-1}(X)} \bigr)_2 \neq 0$, so the previous paragraph implies
  that $\len\bigl( \pi_{k-1}(X) \bigr) = 1$.  Since
  Proposition~\ref{p:linPart} establishes that
  $\len(X) \leqslant (k-1) + \len\bigl( \pi_{k-1}(X) \bigr)$, we conclude that
  $k \geqslant \len(X)$.
\end{proof}

We first show that the inequality in Theorem~\ref{t:2} may fail for a
reducible variety.

\begin{example}[Bounds for a reducible variety]
  The variety $X \subset \mathbb{P}^2$ determined by the monomial ideal
  $\ideal{x_0 x_1, x_0 x_2} = \ideal{x_0} \cap \ideal{x_1, x_2}$ is just the
  union of the $x_0$\nobreakdash-axis and the point $[1:0:0]$.  Since the
  Betti table of its homogeneous coordinate ring is
  \[
    \begin{array}[b]{c|ccccccccccc}
      _{i} \backslash^{j} & 0 & 1 & 2 \\ \hline
      0 & 1 & 0 & 0 \\[-2pt]
      1 & 0 & 2 & 1      
    \end{array} \; ,
  \]
  we deduce that $\len(X) = 2$.  On the other hand, the rational map given by
  projecting away from the point $[1 : 0 : 0] \in X$ surjects onto
  $\mathbb{P}^1$.  Hence, the ideal of the image contains no quadratic
  polynomials, so we have $\qp(X) = 1 < 2 = \len(X)$.
\end{example}

The next two examples demonstrates that the inequality in Theorem~\ref{t:2}
can be strict.  They also answer Question~5.8 in \cite{HK15} negatively.

\begin{example}[Bounds for general canonical curves]
  \label{e:strict}
  Suppose that $X \subset \mathbb{P}^{g-1}$ is a general canonical curve of
  genus $g$ and set $k \coloneqq \qp(X)$.  As in Example~\ref{e:genCan}, the
  Riemann--Roch Theorem implies that
  $\dim_{\mathbb{C}} (I_X)_2 = \binom{g+1}{2} - 3g + 3$.  Since
  Corollary~\ref{c:qpIneq} gives $(g+1)g - 6g + 6 \leqslant 2k(g-2) -k(k-1)$,
  we obtain the lower bound
  $k \geqslant \bigl\lceil g - \frac{3}{2} - \frac{1}{2}\sqrt{8g - 15}
  \bigr\rceil$.  Furthermore, Green's Conjecture, which is explained in
  Section~9B of \cite{Eis} and proven in \cite{Voi}, establishes that
  $\green(X) = \bigl\lceil \frac{1}{2}(g-2) \bigr\rceil - 1$ and
  $\len(X) = g - 3 - \green(X) = \bigl\lceil \frac{1}{2} (g-2) \bigr\rceil$.
  Thus, we have $\qp(X) > \len(X)$ for all $g \geqslant 10$.
\end{example}

\begin{remark}
  Repurposing Example~\ref{e:strict}, we see that there exists a curve
  $X \subset \mathbb{P}^n$ and a general point $p \in X$ such that
  $\len(X) = \len \bigl( \pi_{\{ p \}}(X) \bigr)$.  Indeed, some inner
  projection of a general canonical curve of genus at least ten must yield a
  curve with the desired properties.
\end{remark}

\begin{example}[Bounds for curves of high degree]
  \label{e:high}
  Suppose that $d \gg g$ and $X$ is a smooth irreducible complex curve of
  genus $g$ and gonality $\delta$ embedded by a complete linear series of
  degree $d$ in $\mathbb{P}^n$.  Corollary~8.4 in \cite{Eis} shows that
  $n = d-g$ and $\dim_{\mathbb{C}} (I_X)_2 = \binom{d-g+2}{2} - (2d-g+1)$.
  Setting $k \coloneqq \qp(X)$, Corollary~\ref{c:qpIneq} gives
  \[
    (d-g+2)(d-g+1) - 2(2d-g+1) \leqslant 2k(d-g-1) - k(k-1) \, ,
  \]
  so we obtain
  $k \geqslant \bigl\lceil d - g - \frac{1}{2} - \frac{1}{2}\sqrt{8g +1}
  \bigr\rceil$.  Moreover, the Gonality Conjecture, which is discussed in
  Section~8C of \cite{Eis} and proven in \cite{EL}, asserts that
  $\len(X) = d-g-\delta$.  Therefore, the hypothesis that
  $2 \delta > 1 + \sqrt{8g+1}$ implies that $k > \len(X)$.  As already
  observed in Example~\ref{e:genCan}, the gonality of a general curve is
  $\bigl\lceil \frac{1}{2}(g+2) \bigr\rceil$, so we have the strict inequality
  $\qp(X) > \len(X)$ whenever $X$ is a general curve of genus at least $7$.
\end{example}

We close this section with a curious relationship between three of our
favourite numerical invariants of an irreducible complex subvariety.

\begin{proposition}
  \label{p:qpEq}
  Let $X \subseteq \mathbb{P}^n$ be a non-degenerated irreducible complex
  subvariety.  If there exists a variety $X' \subseteq \mathbb{P}^n$ of
  minimal degree such that $X \subseteq X'$ and $\qp(X) = \qp(X')$, then we
  have $\len(X) = \len(X')$.  Under the additional hypothesis that $X$ is
  totally real, we also have $\py(X) = \py(X')$.
\end{proposition}

\begin{proof}
  Since $X'$ is a variety of minimal degree, Corollary~\ref{c:mindeg} proves
  that $1 + \dim(X') = \py(X')$ and Theorem~\ref{t:maxqp} shows that
  $\qp(X') = \codim(X')$.  Hence, Lemma~\ref{l:inclPy} and Theorem~\ref{t:3}
  give
  \begin{align*}
    1 + \dim(X')
    = \py(X') 
    \geqslant \py(X) 
    &\geqslant n + 1 - \qp(X) \\
    &= n + 1 - \qp(X') 
    = n + 1 - \codim(X') = 1 +\dim(X')
  \end{align*}
  which shows that $\py(X) = \py(X')$.  As $X'$ is a variety of minimal
  degree, Corollaries~A2.62--A2.64 in \cite{Eis} also imply that
  $\len(X') = \codim(X') = \qp(X')$.  Given the inclusion $X \subseteq X'$,
  Corollary~1.28 in \cite{AN} asserts that $\len(X) \geqslant \len(X')$.
  Theorem~\ref{t:2} yields
  $\len(X') = \qp(X') = \qp(X) \geqslant \len(X) \geqslant \len(X')$ which
  demonstrates that $\len(X') = \len(X)$.
\end{proof}

%%%%%%%%%%%%%%%%%%%%%%%%%%%%%%%%%%%%%%%%%%%%%%%%%%%%%%%%%%%%%%%%%%%%%%%%%%%%%%
\section{Toric applications}
\label{s:toric}

\noindent
In this closing section, we refine our estimates on quadratic persistence for
projective toric subvarieties.  Notably, we compute the quadratic persistence
for any Veronese embedding of the projective plane and for the embedded toric
variety corresponding to any sufficiently tall lattice prism.

For a nested pair $X \subseteq X'$ of irreducible complex varieties,
Part~\emph{i} of Lemma~\ref{l:qpBasic} establishes the inequality
$\qp(X) \geqslant \qp(X')$.  Our initial goal is to show that the opposite
inequality holds in a special situation.  To elucidate this partial converse,
we devise a new kind of transversality.  Given a finite set $\Gamma'$ of
closed points in $X' \subseteq \mathbb{P}^n$ spanning a
$(k-1)$\nobreakdash-plane, we write
$\pi_{\Gamma'} \colon \mathbb{P}^n \dashrightarrow \mathbb{P}^{n-k}$ for the
linear projection away from $\Span(\Gamma')$; see Section~\ref{s:qp}.

\begin{definition}
  Let $X' \subseteq \mathbb{P}^n$ be an irreducible complex subvariety.  A
  subvariety $X \subseteq X'$ is \define{transverse to general inner
    projections} if, for all $0 \leqslant k \leqslant \dim(X')$ and all
  subsets $\Gamma'$ of $k$ general closed points in $X'$, there exists a
  subset $\Gamma$ of $k$ general closed points in $X$ such that the image
  $\pi_{\Gamma}(X')$ is projectively equivalent to the image
  $\pi_{\Gamma'}(X')$.
\end{definition}

This definition captures those nested pairs of subvarieties for which the
points in the smaller variety are sufficient to compute the quadratic
persistence of the larger variety.

\begin{lemma}
  Let $X' \subseteq \mathbb{P}^n$ be an irreducible complex subvariety.  If
  $X \subseteq X'$ is transverse to general inner projections, then the
  quadratic persistence of $X'$ is equal to the smallest cardinality of a
  finite set $\Gamma$ of general closed points in $X$ such that the ideal
  $I_{\pi_{\Gamma}(X')}$ contains no quadratic polynomials.
\end{lemma}

\begin{proof}
  By Part~\emph{iv} of Lemma~\ref{l:qpGen}, the quadratic persistence
  $\qp(X')$ is the smallest $k \in \mathbb{N}$ for which there exists a finite
  set $\Gamma'$ of general closed points in $X'$ such that $k = \abs{\Gamma'}$
  and the ideal $I_{\pi_{\Gamma'}(X')}$ contains no quadratic polynomials.
  Since $X$ is transverse to general inner projections, there exists a subset
  $\Gamma$ of general closed points in $X$ such that image $\pi_{\Gamma}(X')$
  is projectively equivalent to the image $\pi_{\Gamma'}(X')$.  Thus, the
  ideal $I_{\pi_{\Gamma}(X')}$ contains no quadratic polynomials which
  completes the proof.
\end{proof}

To relate the number of quadratic polynomials in the homogeneous ideals of $X$
and $X'$, it is convenient to have the following notation.

\begin{definition}
  For the nested sequence $X \subseteq X' \subseteq \mathbb{P}^n$ of complex
  subvarieties, the \define{quadratic residual} of $X$ in $X'$ is defined to
  be the integer
  $\qr(X,X') \coloneqq \dim_{\mathbb{C}}(I_X)_2 - \dim_{\mathbb{C}}
  (I_{X'})_2$.
\end{definition}

Like in Example~3.1 of \cite{Har}, a variety $X' \subseteq \mathbb{P}^n$ is a
\define{cone} if there exists a proper subvariety $X$ and a closed point
$q \in X'$ not lying on $X$ such that $X'$ is the union of the lines
$\Span(\{q, p \})$ spanned by the point $q \in X'$ and the points $p \in X$.
Every such point $q$ is a \define{vertex} of the cone $X'$.  Having collected
the requisite definitions, we now bound the quadratic persistence from above.

\begin{theorem}
  \label{t:trans}
  Let $X \subset \mathbb{P}^n$ be a non-degenerate irreducible complex
  subvariety.  Suppose that $X' \subseteq \mathbb{P}^n$ is a cone containing
  $X$ such that $\dim(X') = 1 + \dim(X)$ and, for a vertex $q \in X'$, we have
  $\pi_{\{q\}}(X') = \pi_{\{q\}}(X)$.  Assuming that $X \subset X'$ is also
  transverse to general inner projections, we obtain the inequality
  $\qp(X) \leqslant \max \{ \qp(X'), \qr(X,X') \}$.
\end{theorem}

\begin{proof}
  Set $k' \coloneqq \qp(X')$. Part~\emph{iv} of Lemma~\ref{l:qpGen} implies
  that, for a general subset $\Gamma'$ of $k'$ closed points in $X'$, the
  ideal $I_{\pi_{\Gamma'}(X')}$ contains no quadratic polynomials.  Since
  $X \subset X'$ is transverse to general inner projections, there exists a
  subset $\Gamma$ of $k'$ closed points in $X$ such that the ideal
  $I_{\pi_{\Gamma}(X')}$ also contains no quadratic polynomials.  If
  necessary, enlarge the subset $\Gamma$, by appending additional general
  closed points in $X$, to ensure that $\abs{\Gamma} \geqslant \qr(X,X')$.  We
  now claim that the homogeneous ideal $I_{\pi_{\Gamma}(X)}$ contains no
  quadratic polynomials.

  To prove this claim, fix an affine representative
  $\widetilde{p} \in \mathbb{A}^{n+1}$, for each closed point $p \in X$, and
  consider the map
  $\operatorname{del}_X \colon (I_X)_2 \to \prod_{p \in X} (T_{X',p} /
  T_{X,p})^*$ defined by
  $\operatorname{del}_X(f) \coloneqq \bigl( \nabla f(\widetilde{p}) \; \big|
  \; p \in X \bigr)$.  We first show that the kernel of this map is
  $(I_{X'})_2$.  The variety $X$ cannot be contained in the singular locus of
  $X'$, because the line corresponding to nonsingular point in $X$ is
  nonsingular in $X'$.  Hence, at a general closed point $p \in X$, the
  tangent space $T_{X',p}$ is naturally isomorphic to
  $T_{X,p} \oplus \Span(\{p,q\})$.  If
  $f \in \operatorname{Ker}(\operatorname{del}_X)$, then the gradient of $f$
  evaluated at the point $\tilde{p}$ is orthogonal to the line
  $\Span(\{p,q\})$.  Since $X$ is non-degenerate, it follows that $f$ vanishes
  to order at least to $2$ at the vertex $q$, so our assumption that
  $\pi_{\{q\}}(X) = \pi_{\{q\}}(X')$ guarantees that $f \in (I_{X'})_2$.  From
  our characterization of the kernel, we see that the image of
  $\operatorname{del}_X$ has dimension $\qr(X,X')$.  Therefore, we deduce that
  $(I_{\pi_{\Gamma}(X)})_2 = 0$ and $k' \geqslant \qp(X)$.
\end{proof}

\begin{remark}
  \label{r:qp=}
  Under the additional hypothesis that $\qr(X,X') \leqslant \qp(X')$,
  Part~\emph{ii} of Lemma~\ref{l:qpBasic} and
  Theorem~\ref{t:trans} combine to prove that $\qp(X) = \qp(X')$.
\end{remark}

To apply Theorem~\ref{t:trans}, we need a better tool for recognizing
subvarieties that are transverse to general inner projections.  The next lemma
and proposition forge such a tool.

\begin{lemma}
  \label{l:scale}
  Let $X' \subseteq \mathbb{P}^n$ be an irreducible complex subvariety that is
  a cone with vertex $q \in X'$.  For any positive integer $k$ and closed
  points
  $p_1^{}, p_2^{}, \dotsc, p_k^{}, p'_1, p'_2, \dotsc, p'_k \in X' \setminus
  \{q\}$ such that $X'$ is not contained in the linear space
  $\Span(\{q,p_1, p_2, \dotsc, p_k\})$ and
  $q \in \Span(\{p_{\!j}^{}, p'_{\!j} \})$ for all
  $1 \leqslant j \leqslant k$, the inner projections
  $\pi_{\{ p_1^{}, p_2^{}, \dotsc, p_k^{} \}}(X')$ and
  $\pi_{\{ p'_1, p'_2, \dotsc, p'_k\}}(X')$ are projectively equivalent.
\end{lemma}

\begin{proof}
  Since our hypothesis include the conditions
  $q \in \Span(\{p_{\!j}^{}, p'_{\!j}\})$ and $q \not\in \{ p_{\!j}^{}, p'_{\!j} \}$,
  we see that
  $\Span(p_1^{}, p_2^{}, \dotsc, p_k^{}, q) = \Span(p'_1, p'_2, \dotsc, p'_k,
  q)$.  For each $q' \in X' \setminus \Span(q, p_1, p_2, \dotsc, p_k)$,
  consider the line $L_{q'} = \Span(\{q,q'\})$.  The union of all $L_{q'}$
  covers a dense subset of $X'$, because $X'$ is a cone.  By fixing a linear
  subspace $\mathbb{P}^{n-k}$ that is complementary to both
  $\Span(p_1^{}, p_2^{}, \dotsc, p_k^{})$ and
  $\Span(p'_1, p'_2, \dotsc, p'_k)$, we deduce that
  $\pi_{\{ p_1^{}, p_2^{}, \dotsc, p_k^{} \}}(L_{q'}) = \pi_{\{ p'_1, p'_2,
    \dotsc, p'_k\}}(L_{q'})$.
\end{proof}

\begin{proposition}
  \label{p:projEq}
  Let $X \subset \mathbb{P}^n$ be an irreducible complex subvariety and let
  $X' \subseteq \mathbb{P}^n$ be a cone containing $X$.  If
  $\pi_{\{q\}}|_X \colon X \dashrightarrow \pi_{\{q\}}(X)$ is birational map
  and $\pi_{\{q\}}(X)$ is projectively equivalent to $\pi_{\{q\}}(X')$, then
  the subvariety $X \subseteq X'$ is transverse to general inner projections.
\end{proposition}

\begin{proof}
  Let $\tau \colon \pi_{\{q\}}(X) \dashrightarrow X$ be the inverse of the
  birational map $\pi_{\{q\}}|_X \colon X \dashrightarrow \pi_{\{q\}}(X)$.  If
  $p_1, p_2, \dotsc, p_k$ are general closed points in $X'$, then their images
  $\pi_{\{q\}}(p_1), \pi_{\{q\}}(p_2), \dotsc, \pi_{\{q\}}(p_k)$ avoid the
  indeterminacy locus of $\tau$, so set
  $p'_{\!j} \coloneqq \tau \bigl( \pi_{\{q\}}(p_{\!j}^{}) \bigr)$ for
  $1 \leqslant j \leqslant k$.  By construction, we have
  $q \in \Span(\{p_{\!j}^{}, p'_{\!j} \})$ and $q \not\in \{p_{\!j}^{}, p'_{\!j} \}$
  for all $1 \leqslant j \leqslant k$.  Hence, Lemma~\ref{l:scale} establishes
  that $\pi_{\{ p_1^{}, p_2^{}, \dotsc, p_k^{} \}}(X')$ and
  $\pi_{\{ p'_1, p'_2, \dotsc, p'_k\}}(X')$ are projectively equivalent, which
  proves that $X \subseteq X'$ is transverse to general inner projections.
\end{proof}

We illustrate the power of Theorem~\ref{t:trans} and
Proposition~\ref{p:projEq} with a family of examples.

\begin{example}[The quadratic persistence of the Veronese embeddings of
  $\mathbb{P}^2$]
  \label{e:qpVer}
  For all $j \geqslant 2$, the map
  $\nu_{\!j} \colon \mathbb{P}^2 \to \mathbb{P}^{\binom{j+2}{2}-1}$ is defined
  by $[x_0 : x_1 : x_2] \mapsto [x_0^j : x_0^{j-1} x_1^{} : \dotsb : x_2^j]$.
  We claim that
  \[
    \qp \bigl( \nu_{\!j} (\mathbb{P}^2) \bigr) = \tbinom{j+1}{2} \, .
  \]
  To prove this, we proceed by induction on $j$.  In the base case, the
  Veronese surface $\nu_2(\mathbb{P}^2) \subset \mathbb{P}^5$ is a variety of
  minimal degree, so Theorem~\ref{t:maxqp} gives
  $\qp \bigl( \nu_2 (\mathbb{P}^2) \bigr) = \codim \bigl( \nu_2 (\mathbb{P}^2)
  \bigr) = 3 = \binom{2+1}{2}$.  For any $j > 2$, the embedded toric surface
  $\nu_{\!j}(\mathbb{P}^2) \subset \mathbb{P}^{j(j+3)/2}$ corresponds to the
  lattice triangle
  $T_{\!j} \coloneqq \conv\{\mathbf{0}, j \, \mathbf{e}_1, j \, \mathbf{e}_2\}
  \subset \mathbb{R}^2$ where $\mathbf{e}_1, \mathbf{e}_2$ denotes the
  standard basis for $\mathbb{R}^2$; see Example~2.3.15 in \cite{CLS}.
  Consider the following sequence of inner projections: for $i$, decreasing by
  $1$ from $j$ to $1$, project away from the torus-invariant point
  corresponding to the lattice point $(0,i)$.  The final embedded projective
  toric variety corresponds to the lattice polytope
  $P \coloneqq \conv\{ \mathbf{0}, T_{\!j-1} + \mathbf{e}_1\}$ and Part~\emph{v}
  of Lemma~\ref{l:qpGen} shows that
  $\qp \bigl( \nu_{\!j} (\mathbb{P}^2) \bigr) \leqslant j + \qp(X_{P \cap
    \smash{\mathbb{Z}^2}})$.

  We next verify that
  $\qp(X_{P \cap \smash{\mathbb{Z}^2}}) = \qp \bigl( \nu_{\!j-1} (\mathbb{P}^2)
  \bigr)$.  Let $\mathbf{e}_0, \mathbf{e}_1, \mathbf{e}_2$ denote the standard
  basis for $\mathbb{R}^3$ and set
  $P' \coloneqq \conv\{ \mathbf{e}_0, T_{\!j-1} + \mathbf{e}_1 \} \subset
  \mathbb{R} \times \mathbb{R}^2 \cong \mathbb{R}^3$. The coordinate
  projection $\mathbb{R} \times \mathbb{R}^2 \to \mathbb{R}^2$ defines a
  bijection between the lattice points in $P$ and $P'$ and establishes that
  the associated toric varieties are nested in the same ambient projective
  space.  Since the lattice polytope $P'$ is a pyramid, its associated
  embedded projective toric variety $X_{\smash{P' \cap \mathbb{Z}^3}}$ is a
  cone whose vertex corresponds to the lattice point $\mathbf{e}_0 \in P'$, so
  Proposition~\ref{p:projEq} shows that the subvariety
  $X_{P \cap \smash{\mathbb{Z}^2}} \subset X_{\smash{P' \cap \mathbb{Z}^3}}$
  is transverse to general inner projections. Applying Part~\emph{ii} of
  Lemma~\ref{l:qpBasic} and Theorem~\ref{t:trans}, we obtain the inequalities
  $\qp(X_{\smash{P' \cap \mathbb{Z}^3}}) \leqslant \qp(X_{P \cap
    \smash{\mathbb{Z}^2}}) \leqslant \max\{ \qp(X_{\smash{P' \cap
      \mathbb{Z}^3}}), \qr(X_{P \cap \smash{\mathbb{Z}^2}}, X_{\smash{P' \cap
      \mathbb{Z}^2}}) \}$.  We deduce that the lattice polytopes $P$ and $P'$
  are normal from Corollary~2.2.13 in \cite{CLS}.  Hence, combining
  Theorem~5.4.8 and Theorem~9.2.3 in \cite{CLS} with the theory of Ehrhart
  polynomials (see Section~9.4 of \cite{CLS}) yields
  \[
    \qr(X_{P \cap \smash{\mathbb{Z}^2}}, X_{\smash{P' \cap \mathbb{Z}^3}}) =
    \abs{2P' \cap \mathbb{Z}^3} - \abs{2P \cap \mathbb{Z}^2} = \left(
      \tbinom{2j}{2} + \tbinom{j+1}{2} + 1 \right) - \left( \tbinom{2j}{2} + j
      + 1 \right) = \tbinom{j}{2} \, .
  \]
  Because $P'$ is a pyramid over the polygon $T_{\!j-1} + \mathbf{e}_1$, we also
  have
  $\qp(X_{\smash{P' \cap \mathbb{Z}^3}}) = \qp(X_{(T_{\!j-1} + \mathbf{e}_1)
    \cap \smash{\mathbb{Z}^2}})$.  Thus, the induction hypothesis gives
  $\qp(X_{(T_{\!j-1} + \mathbf{e}_1) \cap \smash{\mathbb{Z}^2}}) =
  \qp(X_{T_{\!j-1} \cap \smash{\mathbb{Z}^2}}) = \qp\bigl( \nu_{\!j-1}
  (\mathbb{P}^2) \bigr) = \binom{j}{2}$, so we conclude that
  $\qp(X_{P \cap \smash{\mathbb{Z}^2}}) = \binom{j}{2} = \qp \bigl( \nu_{\!j-1}
  (\mathbb{P}^2) \bigr)$.

  The inequality at the end of the first paragraph together with the equality
  in the second paragraph prove that
  $\qp \bigl( \nu_{\!j} (\mathbb{P}^2) \bigr) \leqslant \binom{j+1}{2}$.  For
  the complementary lower bound, observe that
  \[
    \dim_{\mathbb{C}} \bigl( I_{\nu_{\!j} (\mathbb{P}^2)} \bigr)_2^{} =
    \binom{\binom{j+2}{2} + 1}{2} - \binom{2j+2}{2} =
    \sum_{i=2j+2}^{\binom{j+2}{2}} i = \sum_{i=j+2}^{\binom{j+2}{2}} i -
    \sum_{i=j+2}^{2j+1} i =\sum_{i=j+2}^{\binom{j+2}{2}} (i-3)
  \]
  and the right side is the sum of the codimension of the varieties obtained
  by successively projecting $\nu_{\!j} (\mathbb{P}^2)$ away from a point
  $\binom{j+1}{2}$ times.  Thus, Part~\emph{i} of Lemma~\ref{l:qpBasic} shows
  that we need to project away from at least $\binom{j+1}{2}$ points to
  eliminate all quadratic polynomials, so
  $\qp \bigl( \nu_{\!j} (\mathbb{P}^2) \bigr) \geqslant \binom{j+1}{2}$.
\end{example}

\begin{remark}
  The techniques developed in \cite{Sch17} yield a different proof that
  $\qp \bigl( \nu_j (\mathbb{P}^2) \bigr) = \tbinom{j+1}{2}$.  This completely
  independent approach hinges on knowning the Hilbert function for the square
  of the vanishing ideal for general closed points in $\mathbb{P}^2$; see
  Proposition~4.8 in \cite{IK}.
\end{remark}

\begin{remark}
  The tactic employed in Example~\ref{e:qpVer} to realize a toric variety as a
  subvariety transverse to general inner projections generalizes.  For a
  lattice polytope $P \subset \mathbb{R}^{d}$ and a vertex
  $\mathbf{v} \in P \cap \mathbb{Z}^d$, set
  $P' \coloneqq \conv\{ \mathbf{v} + \mathbf{e}_0, (P \cap \mathbb{Z}^d)
  \setminus \mathbf{v} \} \subset \mathbb{R} \times \mathbb{R}^d \cong
  \mathbb{R}^{d+1}$.  Using Proposition~\ref{p:projEq}, one may verify that
  the toric inclusion
  $X_{P \cap \smash{\mathbb{Z}^d}} \subset X_{\smash{P' \cap
      \mathbb{Z}^{d+1}}}$ is always transverse to general inner projections.
\end{remark}

Our formula for the quadratic persistence of the toric surface
$\nu_{\!j}(\mathbb{P}^2) \subset \mathbb{P}^{j(j+3)/2}$ also produces bounds on
its Pythagoras number, re-proving Theorem~3.6 in \cite{Sch17}

\begin{example}[Pythagoras numbers for the Veronese embeddings of
  $\mathbb{P}^2$]
  \label{e:PP2Low}
  Combining Example~\ref{e:qpVer} and Theorem~\ref{t:3} gives
  $\py \bigl( \nu_{\!j}(\mathbb{P}^2) \bigr) \geqslant \binom{j+2}{2} -
  \binom{j+1}{2} = j+1$. Since Example~\ref{e:PP2Up} shows that
  $\py \bigl( \nu_{\!j}(\mathbb{P}^2) \bigr) \leqslant j+2$, we confirm that
  $\py \bigl( \nu_{\!j}(\mathbb{P}^2) \bigr)$ is either $j+1$ or $j+2$.
\end{example}

We next calculate the quadratic persistence for projective toric subvarieties
arising from a special class of polytopes.  For all positive
$k \in \mathbb{Z}$ and any lattice polytope $P \subset \mathbb{R}^d$, the
prism $P \times [0,k] \subset \mathbb{R}^{d+1}$ is also a lattice polytope.
The ensuing proposition shows that a rational normal scroll containing the
toric variety $X_{(P \times [0,k]) \cap \smash{\mathbb{Z}^{d+1}}}$ determines
its quadratic persistence for all large $k$.

\begin{proposition}
  \label{p:prism}
  For any lattice polytope $P \subset \mathbb{R}^d$ having dimension greater
  than one and any positive integer $k$ greater than or equal to
  $\frac{1}{\dim(P)-1} \abs{P \cap \smash{\mathbb{Z}^d}} - 1$, the quadratic
  persistence of the projective toric subvariety associated to the prism
  $P \times [0,k]$ equals $k \, \abs{P \cap \smash{\mathbb{Z}^d}} - 1$.
  Moreover, we also have
  $\py(X_{(P \times [0,k]) \cap \smash{\mathbb{Z}^{d+1}}}) = 1 + \abs{P \cap
    \mathbb{Z}^d}$ and
  $\len(X_{(P \times [0,k]) \cap \smash{\mathbb{Z}^{d+1}}}) = k \, \abs{P \cap
    \mathbb{Z}^d} - 1$.
\end{proposition}

\begin{proof}
  We proceed by induction on $\abs{P \cap \smash{\mathbb{Z}^d}}$.  For the
  base case, it suffices to consider a standard simplex. If
  $\mathbf{e}_1, \mathbf{e}_2, \dotsc, \mathbf{e}_d$ denotes the standard
  basis for $\mathbb{R}^d$, then we have
  $P = \conv\{ \mathbf{0}, \mathbf{e}_1, \mathbf{e}_2, \dotsc, \mathbf{e}_d
  \}$.  For any positive integer $k$, the corresponding toric subvariety
  $X_{(P \times [0,k]) \cap \smash{\mathbb{Z}^{d+1}}}$ is the Segre embedding
  of the product $\mathbb{P}^d \times \nu_{k}(\mathbb{P}^1)$ in
  $\mathbb{P}^{kd+d+k}$, where the factor
  $\nu_{k}(\mathbb{P}^1) \subset \mathbb{P}^k$ is the rational normal curve of
  degree $k$. This variety is itself a rational normal scroll, so
  Theorem~\ref{t:maxqp} establishes that
  \[
    \qp(X_{(P \times [0,k]) \cap \smash{\mathbb{Z}^{d+1}}}) = \codim(X_{(P
      \times [0,k]) \cap \smash{\mathbb{Z}^{d+1}}}) = (kd+d+k) - (d+1) = k \,
    \abs{P \cap \smash{\mathbb{Z}^d}} - 1 \, .
  \]

  Now, suppose that $P \subset \mathbb{R}^d$ is an arbitrary lattice polytope
  and assume that the positive integer $k$ satisfies
  $k \geqslant \frac{1}{\dim(P)-1} \abs{P \cap \smash{\mathbb{Z}^d}} - 1$.
  Corollary~\ref{c:lines} shows that the embedded projective toric variety
  $X_{(P \times [0,k]) \cap \smash{\mathbb{Z}^{d+1}}}$ is contained in a
  rational normal scroll $X_{\smash{P'} \cap \smash{\mathbb{Z}^{m}}}$ whose
  dimension $m \coloneqq \abs{P \cap \mathbb{Z}^d}$ is equal to the number of
  parallel lines needed to cover all of the lattice points in the prism
  $P \times [0,k]$.  Hence, Part~\emph{ii} of Lemma~\ref{l:qpBasic} and
  Theorem~\ref{t:maxqp} give the lower bound
  \begin{align*}
    \qp(X_{(P \times [0,k]) \cap \smash{\mathbb{Z}^d}})
    \geqslant \qp(X_{\smash{P' \cap \mathbb{Z}^m}})
    &= \codim(X_{\smash{P' \cap \mathbb{Z}^m}}) \\
    &= \textstyle \bigl( (k+1) \, \abs{P \cap \mathbb{Z}^d} -1 \bigr) - \abs{P
      \cap \mathbb{Z}^d} = k \, \abs{P \cap \mathbb{Z}^d} - 1 \, .
  \end{align*}
  To prove the complementary upper bound, choose a vertex $\mathbf{v} \in P$.
  Set $Q \coloneqq \conv\{ (P \cap \mathbb{Z}^d) \setminus \mathbf{v} \}$, so
  $\dim(P)-1 \leqslant \dim(Q) \leqslant \dim(P)$.  Since
  $\abs{Q \cap \smash{\mathbb{Z}^d}} < \abs{P \cap \smash{\mathbb{Z}^d}}$, the
  induction hypothesis establishes that
  $\qp(X_{(Q \times [0,k]) \cap \smash{\mathbb{Z}^{d+1}}}) = k \, \abs{Q \cap
    \mathbb{Z}^d} - 1$.  We relate this quantity to
  $\qp(X_{(P \times [0,k]) \cap \smash{\mathbb{Z}^{d+1}}})$ via the following
  sequence of inner projections: for $i$, decreasing by $1$ from $k$ to $1$,
  project away from the torus-invariant point corresponding to the lattice
  point $(\mathbf{v}, i)$.  We are moving down the edge of the prism
  $P \times [0,k]$ lying over the vertex $\mathbf{v}$.  The final embedded
  projective toric variety corresponds to
  $Q' \coloneqq \conv\bigl\{ (Q \times [0,k]) \cup \{ (\mathbf{v},0) \}
  \bigr\}$.  We claim that
  $\qp(X_{\smash{Q' \cap \mathbb{Z}^{d+1}}}) = \qp(X_{(Q \times [0,k]) \cap
    \smash{\mathbb{Z}^{d+1}}})$.  This claim together with Part~\emph{v} of
  Lemma~\ref{l:qpGen} would give
  \[
    \qp(X_{(P \times [0,k]) \cap \smash{\mathbb{Z}^{d+1}}}) \leqslant k +
    \qp(X_{\smash{Q' \cap \mathbb{Z}^{d+1}}}) = \textstyle k + k \, \abs{Q
      \cap \mathbb{Z}^d} - 1 = k \, \abs{P \cap \mathbb{Z}^d} - 1
  \]
  as required.  Thus, it only remains to prove the claim.

  To accomplish this, choose a lattice point
  $\mathbf{w} \in P \cap \mathbb{Z}^d$ adjacent to the vertex
  $\mathbf{v} \in P$ such that the primitive vector $\mathbf{v} - \mathbf{w}$
  is parallel to an edge of the polytope $P$ passing through $\mathbf{v}$.
  Consider the pyramid
  $Q'' \coloneqq \conv\bigl\{ (Q \times [0,k] \times 0) \cup \{
  (\mathbf{w},0,1) \} \bigr\} \subset \mathbb{R}^d \times \mathbb{R} \times
  \mathbb{R}$ and the linear projection
  $\theta \colon \mathbb{R}^d \times \mathbb{R} \times \mathbb{R} \to
  \mathbb{R}^{d} \times \mathbb{R}$ defined by
  $(\mathbf{u},y,z) \mapsto (\mathbf{u},y) + z (\mathbf{v} - \mathbf{w}, 0)$.
  By design, the map $\theta$ induces a bijection between the lattice points
  in $Q''$ and $Q'$, so the associated toric varieties are nested in the same
  ambient projective space.  Since the lattice polytope $Q''$ is a pyramid,
  the embedded projective toric variety
  $X_{\smash{Q'' \cap \mathbb{Z}^{d+2}}}$ is a cone and
  Proposition~\ref{p:projEq} shows that the subvariety
  $X_{\smash{Q' \cap \mathbb{Z}^{d+1}}} \subset X_{\smash{Q'' \cap
      \mathbb{Z}^{d+2}}}$ is transverse to general inner projections. Applying
  Part~\emph{ii} of Lemma~\ref{l:qpBasic} and Theorem~\ref{t:trans}, we obtain
  $\qp(X_{\smash{Q'' \cap \mathbb{Z}^{d+1}}}) \leqslant \qp(X_{\smash{Q' \cap
      \mathbb{Z}^{d+1}}}) \leqslant \max\{ \qp(X_{\smash{Q'' \cap
      \mathbb{Z}^{d+2}}}), \qr(X_{\smash{Q' \cap \mathbb{Z}^{d+1}}},
  X_{\smash{Q'' \cap \mathbb{Z}^{d+2}}}) \}$.  Regarding the homogeneous
  coordinates rings of these embedded projective toric varieties as semigroup
  algebras (see Theorem~1.1.7 in \cite{CLS}), we have
  \[
    \qr(X_{\smash{Q' \cap \mathbb{Z}^{d+1}}}, X_{\smash{Q'' \cap
        \mathbb{Z}^{d+2}}}) = \textstyle \abs{Q'' \cap \mathbb{Z}^{d+2} + Q''
      \cap \mathbb{Z}^{d+2}} - \abs{Q' \cap \mathbb{Z}^{d+1} + Q' \cap
      \mathbb{Z}^{d+1}} \, .
  \]
  Partitioning via the last coordinate, we deduce that
  \[
    \textstyle \abs{Q'' \cap \mathbb{Z}^{d+2} + Q'' \cap \mathbb{Z}^{d+2}} =
    \textstyle \abs{(Q \times [0,k]) \cap \mathbb{Z}^{d+1} + (Q \times
      [0,k]) \cap \mathbb{Z}^{d+1}} + \textstyle \abs{(Q \times [0,k])
      \cap \mathbb{Z}^{d+1}} + 1 \, .
  \]
  Set
  $\mathscr{A} \coloneqq \left\{ \mathbf{u} \in Q \cap \mathbb{Z}^d \;
    \middle| \; \mathbf{u} + \mathbf{v} \not\in Q \right\}$.  For all
  $\mathbf{u} \in Q \cap \mathbb{Z}^d$ and all $i \in \mathbb{Z}$ satisfying
  $0 \leqslant i \leqslant k$, the condition
  $(\mathbf{u},i) + (\mathbf{v},0) \not\in Q' \cap \mathbb{Z}^{d+1} + Q' \cap
  \mathbb{Z}^{d+1}$ implies that $\mathbf{u} + \mathbf{v} \not\in Q$, so a
  similar partition gives
  \[
    \textstyle \abs{Q' \cap \mathbb{Z}^{d+1} + Q' \cap \mathbb{Z}^{d+1}} =
    \textstyle \abs{(Q \times [0,k]) \cap \mathbb{Z}^{d+1} + (Q \times
      [0,k]) \cap \mathbb{Z}^{d+1}} + (k+1) \, \abs{\mathscr{A}} + 1 \, .
  \]
  It follows that
  $\qr(X_{\smash{Q' \cap \mathbb{Z}^{d+1}}}, X_{\smash{Q'' \cap
      \mathbb{Z}^{d+2}}}) = (k+1) \bigl( \abs{Q \cap \mathbb{Z}^d} -
  \abs{\mathscr{A}} \bigr)$. Since $Q''$ is a pyramid over the prism
  $Q \times [0,k]$, we also have
  $\qp(X_{\smash{Q'' \cap \mathbb{Z}^{d+2}}}) = \qp(X_{(Q \times [0,k]) \cap
    \smash{\mathbb{Z}^{d+1}}}) = k \, \abs{Q \cap \mathbb{Z}^d} - 1$.  As
  advertised in Remark~\ref{r:qp=}, the additional inequality
  $\qr(X_{\smash{Q' \cap \mathbb{Z}^{d+1}}}, X_{\smash{Q'' \cap
      \mathbb{Z}^{d+2}}}) \leqslant \qp(X_{\smash{Q'' \cap
      \mathbb{Z}^{d+2}}})$ would give the equality
  $\qp(X_{\smash{Q' \cap \mathbb{Z}^{d+1}}}) = \qp(X_{\smash{Q'' \cap
      \mathbb{Z}^{d+2}}})$ and, thereby, prove the claim.  This additional
  inequality is equivalent to
  $\abs{P \cap \mathbb{Z}^d} = \abs{Q \cap \mathbb{Z}^d} + 1 \leqslant (k+1)
  \abs{\mathscr{A}}$.  To estimate the cardinality of $\mathscr{A}$, consider
  a facet $F \subset Q$ that is not a facet of $P$.  For each lattice point
  $\mathbf{u} \in F \cap \mathbb{Z}^d$, we have
  $\mathbf{u} + \mathbf{v} \not\in Q$, so
  $\abs{\mathscr{A}} \geqslant \abs{F \cap \mathbb{Z}^d}$.  Because $F$ is a
  lattice polytope of dimension $\dim(Q) - 1 \geqslant \dim(P) - 1$, we infer
  that $\abs{\mathscr{A}} \geqslant \dim(P) -1$.  Therefore, the hypothesis
  that $k \geqslant \frac{1}{\dim(P)-1} \abs{P \cap \mathbb{Z}^d} - 1$
  guarantees that additional inequality holds.  Finally, using the rational
  normal scroll $X_{\smash{P'} \cap \smash{\mathbb{Z}^{m}}}$,
  Proposition~\ref{p:qpEq} proves that
  $\py(X_{(P \times [0,k]) \cap \smash{\mathbb{Z}^{d+1}}}) = 1 + \abs{P \cap
    \mathbb{Z}^d}$ and
  $\len(X_{(P \times [0,k]) \cap \smash{\mathbb{Z}^{d+1}}}) = k \, \abs{P \cap
    \mathbb{Z}^d} -1$.
\end{proof}

We draw attention to an application of Proposition~\ref{p:prism} in which the
hypothesis on $k$ is vacuous.

\begin{example}[Special Segre--Veronese embeddings of
  $\mathbb{P}^d \times \mathbb{P}^1 \times \mathbb{P}^1$]
  \label{e:segVer}
  Fix three positive integer $d, j, k \in \mathbb{N}$ with $k \geqslant j$,
  let $\mathbf{e}_1, \mathbf{e}_2, \dotsc, \mathbf{e}_d$ denote the standard
  basis for $\mathbb{R}^d$, and consider the lattice polytope
  $P \coloneqq \conv\{ \mathbf{0}, \mathbf{e}_1, \mathbf{e}_2, \dotsc,
  \mathbf{e}_d \} \times [0,j]$.  The corresponding toric variety
  $X_{P \times [0,k]}$ is the Segre embedding of the triple product
  $\mathbb{P}^d \times \nu_{\!j}(\mathbb{P}^1) \times \nu_{k}(\mathbb{P}^1)$
  into $\mathbb{P}^{(d+1)(j+1)(k+1) -1}$, so Proposition~\ref{p:prism} gives
  $\qp(X_{P \times [0,k]}) = k(d+1)(j+1) - 1 = \len(X_{P \times [0,k]})$ and
  $\py(X_{P \times [0,k]}) = (d+1)(j+1) + 1$.
\end{example}

\subsection*{Acknowledgements}

We thank Roser Homs Pons, Michael Kemeny, Alessandro Oneto, Simon Telen, and
an anonymous referee for their suggestions.  Grigoriy Blekherman was partially
supported by NSF grant DMS-1352073, Gregory G.~Smith was partially supported
by NSERC and Knut \& Alice Wallenberg Foundation, and Mauricio Velasco was
partially supported by the FAPA funds from Universidad de los Andes.

%%%%%%%%%%%%%%%%%%%%%%%%%%%%%%%%%%%%%%%%%%%%%%%%%%%%%%%%%%%%%%%%%%%%%%%%%%%%%%

\raggedright

\end{document}